\documentclass[12pt]{amsart}
\usepackage[normalem]{ulem}

\usepackage[hmargin=2.5cm,bmargin=2.5cm,tmargin=3cm]{geometry}

\usepackage[T1]{fontenc}
\usepackage[utf8]{inputenc}
\usepackage[english]{babel}
\usepackage{amsmath,amssymb}
\usepackage{amsthm}
\usepackage{amsfonts}
\usepackage{enumitem}
\usepackage{color}
\usepackage{graphicx}
\usepackage{multirow}
\usepackage{calc}
\usepackage[colorlinks=true,linkcolor=blue]{hyperref}
\usepackage{tabularx}
\usepackage{caption}
\usepackage{tikz}
\usepackage{tikz-3dplot}
\usetikzlibrary{calc,angles,quotes}
\usetikzlibrary{through}
\usetikzlibrary{cd}
\usepackage{xcolor}
\usepackage{pgfplots}

\allowdisplaybreaks

\newcommand{\R}{\ensuremath{\mathbb{R}}}
\newcommand{\C}{\ensuremath{\mathbb{C}}}
\newcommand{\N}{\ensuremath{\mathbb{N}}}

\newcommand{\diff}{\ensuremath{\mathop{}\!\mathrm{d}}}
\DeclareMathOperator{\dist}{dist}
\DeclareMathOperator{\Id}{Id}

\DeclareMathOperator{\new}{new}
\DeclareMathOperator{\old}{old}
\DeclareMathOperator{\diam}{diam}
\DeclareMathOperator{\gen}{gen}

\newcommand{\Deg}[2]{\ensuremath{d_{#2}^{[#1]}}}

\newcommand{\init}{\ensuremath{\mathrm{init}}}
\newcommand{\term}{\ensuremath{\mathrm{term}}}
\newcommand{\confclass}{\ensuremath{\mathfrak c}}
\newcommand{\rmetric}{\ensuremath{\mathfrak g}}
\newcommand{\Rmetric}{\ensuremath{\mathfrak h}}

\newtheorem{theorem}{Theorem}[section]
\newtheorem{corollary}[theorem]{Corollary}
\newtheorem{definition}[theorem]{Definition}
\newtheorem{lemma}[theorem]{Lemma}
\newtheorem{proposition}[theorem]{Proposition}
\theoremstyle{definition}
\newtheorem{example}[theorem]{Example}
\newtheorem{remark}[theorem]{Remark}

\newlist{remarklist}{enumerate}{1}
\setlist[remarklist]{
	leftmargin=0cm,
	itemindent=0.7cm,
	label={(\arabic*)},
	itemsep=0.1cm,
}

\numberwithin{equation}{section}

\title[Eigenvalue bounds for the Laplacian on Embedded Metric Graphs]{Upper Eigenvalue Bounds for the Kirchhoff Laplacian on Embedded Metric Graphs}
\author[M.~Pl\"umer]{Marvin Pl\"umer}
\address{Marvin Pl\"umer, Lehrgebiet Analysis, Fakult\"at Mathematik und Informatik, Fern\-Universit\"at in Hagen, 58084 Hagen, Germany}
\email{marvin.pluemer@fernuni-hagen.de}
\thanks{The author was supported by the Deutsche Forschungsgemeinschaft (Grant 397230547). The author also thanks Delio Mugnolo (Hagen) for valueable suggestions and discussions and Pavel Kurasov (Stockholm) for useful ideas in the proof of Corollary~\ref{cor:CompareSpectralGaps}.}



\begin{document}

\begin{abstract}
We derive upper bounds for the eigenvalues of the Kirchhoff Laplacian on a compact metric graph depending on the graph's genus $g$. These bounds can be further improved if $g=0$, i.e. if the metric graph is planar. Our results are based on a spectral correspondence between the Kirchhoff Laplacian and a particular a certain combinatorial weighted Laplacian. In order to take advantage of this correspondence, we also prove new estimates for the eigenvalues of the weighted combinatorial Laplacians that were previously known only in the weighted case.
\end{abstract}

\maketitle
\section{Introduction}

The spectrum of quantum graphs -- self-adjoint differential operators, typically Schr\"odinger operators, defined upon metric graphs --  has been studied very actively in recent years. In this article, we are in particular interested in the spectral properties of the \textit{Kirchhoff Laplacian} \(\Delta\).  If the metric graph is equilateral, i.e.\ all edges have the same length \(1\) the spectral properties of \(\Delta\) are well-known: in \cite{vonBelow:CharEqAs,Cattaneo:SpectContLapl} it was proved that the spectral problem of the Kirchhoff Laplacian on an equilateral metric graph can be reformulated explicitly as a spectral problem of the associated, so-called \textit{normalized Laplacian} \(\mathcal L_{\mathrm{norm}}\). A detailed discussion of the spectral properties of the normalized Laplacian can be found in \cite{Chung:Spec}.  For instance, von Below proved that, if the metric graph is additionally connected and compact, the spectral gaps -- the lowest positive eigenvalues -- of \(\Delta\) and \(\mathcal L_{\mathrm{norm}}\) are related via
	\begin{align}\label{eq:VonBelow}
		\lambda_2(\Delta)=\arccos(1-\lambda_2(\mathcal{L}_{\mathrm{norm}}))^2
	\end{align}
if \(\lambda_2(\mathcal{L}_{\mathrm{norm}})\in [0,2)\). Unfortunately, the techniques leading to \eqref{eq:VonBelow} cannot be extended to arbitrary graphs and no analogous expression for non-equilateral graphs is known. This is why one is particularly interested in estimates for the eigenvalues of \(\Delta\) on general compact metric graphs that depend on combinatorial or metric quantities of the graph. Several bounds on the eigenvalues of \(\Delta\) have been shown for general graphs since the Faber--Krahn-type bounds proved in~\cite{Nicaise:SpecDesRe} for the spectral gap and \cite{Friedlande:ExtrPreop} for arbitrary eigenvalues: in \cite{KennedyKurasovMalenovaMugnolo:OnSpecGap} it was discussed which combinations of specific metric and combinatorial quantities allow for lower and upper bounds on the spectral gap of \(\Delta\); while many of these bounds are sharp, it is known that improved bounds hold for special classes of graphs, like \textit{trees}~\cite{Rohleder:EigEst}, \textit{highly connected graphs}~\cite{band2017quantum,BerkolaikoKennedyKurasovMugnolo:EdgeConn}, or \textit{pumpkin chains}~\cite{borthwick2019sharp}.

In this article, we discuss the spectral properties induced by the lowest \textit{genus} $g$ of an oriented and closed surface the metric graph can be embedded in: if for instance $g=0$, then the metric graph is \textit{planar}, i.e. it can be drawn on the surface of a ball, or equivalently the Euclidean plane without self-crossings. 
More precisely, we are going to derive upper bounds for the eigenvalues of the Kirchhoff Laplacian on metric graphs of arbitrary genus.\\

While planar graphs are fundamental objects of topological graph theory (see for instance \cite{MoharThomassen:GraphsOnSurfaces}), to the best of our knowledge the influence of planarity -- or, more generally, embedding features -- on the spectrum of quantum graphs has never been studied so far. This is in sharp contrast to the theory developed by Spielman and Teng in \cite{SpielmanTeng:SpecPartWork96} and following work \cite{Kelner:SpecPart, BiswaLeeRao:EigBou, KelnerLeePriceTeng:MetrUni} -- we refer especially to their bound
	\begin{equation}\label{eq:intro-spielman-teng}
		\lambda_2(\mathcal{L})\leq 8\frac{d_{\max}}{|V_G|}.
	\end{equation}
on the spectral gap for the unweighted Laplacian \(\mathcal L\) on \textit{combinatorial} planar graphs \cite[Theorem~3.3]{SpielmanTeng:SpecPartWork07} where \(d_{\max}\) is the maximum degree of the vertices of \(G\).  Their work had ground-breaking impact on numerical computing and machine learning and was followed by a vast amount of research~\cite{Kelner:SpecPart, BiswaLeeRao:EigBou, KelnerLeePriceTeng:MetrUni, AminiCohenSteiner:TransferPrinc} extending their results to higher order eigenvalues and graphs of higher genus. As it is going to be important for later discussion, we specifically single out the spectral bound obtained by Amini and Cohen-Steiner in \cite{AminiCohenSteiner:TransferPrinc}: they proved existence of a generic constant \(C>0\) so that
	\begin{align}\label{eq:intro-acs}
		\lambda_k(\mathcal{L}_{\mathrm{norm}})\leq C\frac{d_{\max}(g+k)}{|V_G|},\quad k=1,\ldots,|V_G|.
\end{align}
holds for the $k$-th ordered eigenvalue of the normalized Laplacian \(\mathcal{L}_{\mathrm{norm}}\) on a combinatorial graph of genus $g$ and a generic constant \(C>0\) and -- as we will see in Remark \ref{rem:on-est-genus} -- their estimate can be seen as a normalized higher genus and higher order analogue of the estimate \eqref{eq:intro-spielman-teng}.

Before we are going to state our main results, let us also remark that eigenvalue bounds in dependence of the genus are also well-established in the spectral theory of Laplacians on manifolds: following previous works \cite{Hersch:Quatre, YangYau:EigLapl, Korevaar:UpBound}, it was shown by Hassannezhad \cite{Hassannezhad:Conf} that the \(k\)-th eigenvalue of the Laplace--Beltrami operator \(\Delta_M=-\mathrm{div}_M\circ\nabla_M\) on a Riemannian surface \((M,\mathfrak g)\) of genus \(g\) admits the upper bound
\begin{align}	
	\lambda_k(\Delta_M)\leq C\frac{(g+k)}{\mathrm{Vol}_\mathfrak{g}(M)},\label{eq:Hassannezhad}
\end{align}
where \(C>0\) is a generic constant and \(\mathrm{Vol}_\mathfrak{g}(M)\) denotes the measure of \(M\) with respect to the Riemannian metric \(\mathfrak g\).
\subsection{Statement of the main results and structure of the article}
Let \(\mathcal G=(G,\ell)\) be a compact and connected metric graph with underlying combinatorial graph \(G=(V_G,E_G)\) and a weight function \(\ell\) that assigns to each edge \(e\in E_G\) its length \(\ell_e\). For a vertex \(v\in V_G\) let \(E_v\) denote the set of edges initiating and terminating in \(v\). The Kirchhoff Laplacian \(\Delta_\mathcal G\) on \(\mathcal G\) is the operator acting edgewise as the negative second derivative on the space of functions that are continuous and satisfy Kirchhoff conditions in every vertex \(v\in V_G\); see Section \ref{subsec:prel-metric-graphs} for details. We say that \(\mathcal G\) is of (topological) genus \(g\geq 0\) if so is its underlying combinatorial graph \(G\). If \(\mathcal G\) is of genus \(g\), we shall show that the first \(|V_G|\) eigenvalues \(\lambda_k(\Delta_\mathcal G)\) of \(\Delta_\mathcal G\) satisfy
\begin{equation}\label{eq:intro-estimate-genus1}
		\lambda_k(\Delta_\mathcal G)\leq C\frac{\Deg{\ell}{\max}(g+k)}{\ell_{\min}^2L},\quad k=1,\ldots, |V_G|,
\end{equation}
where \(C>0\) is some generic constant, \(\ell_{\min}=\min_{e\in E_G}\ell_e\) is the \emph{minimal edge length} of \(\mathcal G\), \(L=\sum_{e\in E_G}\ell_{e}\) is the \emph{total length} of \(\mathcal G\) and \(d_{\max}^\ell=\max_{v\in V_G}\sum_{e\in E_v} \ell_e\) is the \emph{maximal degree} of the vertices of \(\mathcal G\) with respect to the edge lengths. We also derive an upper bound for the eigenvalues of \(\Delta_\mathcal G\) of higher order: if \(\mathcal G\) is of genus \(g\) we shall show that
	\begin{equation}\label{eq:intro-estimate-genus2}
		\lambda_k(\Delta_\mathcal G)\leq C\frac{d_{\max}(\beta+k-1)(g+k)}{L^2},\quad k\geq \frac{L}{\ell_{\min}}-\beta+1 ,
	\end{equation}
where additionally to the previously mentioned combinatorial and metric quantities \(d_{\max}\) is the maximal combinatorial degree and \(\beta\) is the \emph{first Betti number} of \(\mathcal G\), i.e., the number of independent cycles in \(\mathcal G\). The estimates \eqref{eq:intro-estimate-genus1} and \eqref{eq:intro-estimate-genus2} can be thought as quantum graph versions of the estimate \eqref{eq:Hassannezhad} and, in fact, \eqref{eq:Hassannezhad} will be an important ingredient for the proof of the estimates. They are going to be proved in Theorem \ref{thm:maintheorem-genus} and Theorem \ref{theo:est-with-dmax} respectively.

We are particularly interested in the choice $g=0$, i.e.\ \(\mathcal G\) is planar. Section~\ref{sec:eigenvboundplanar} will be mainly devoted to proving, in the planar case,  the stronger bound
	\begin{equation}\label{eq:intro-estimate-planar}
		\lambda_2(\Delta_\mathcal G)\leq C\frac{\Deg{\mu}{\max}}{L}
	\end{equation}
for the spectral gap  of \(\Delta_\mathcal G\). Here \(\Deg{\mu}{\max}\) denotes the maximum degree of the vertices of \(\mathcal G\) with respect to the \emph{inverse weight} \(\mu=\frac{1}{\ell}\), that is \(\Deg{\mu}{\max}:=\max_{v\in V}\sum_{e\in E_v}\frac{1}{\ell_e}\). In the planar case, we are able to deliver an explicit estimate on $C$, see Theorem~\ref{thm:maintheorem}. While our estimate on $C$ is itself certainly not optimal, we  show in Remark~\ref{rem:star} that the asymptotic bound
 \[
 \lambda_2(\Delta_\mathcal G)=O\left(\frac{\Deg{\mu}{\max}}{L}\right)\quad\hbox{as}\;\; \frac{\Deg{\mu}{\max}}{L}\to \infty
 \]
is sharp. Let us also point out that the maximum degree \(\Deg{\mu}{\max}\), which to our knowledge has not yet been considered as a parameter in spectral estimates for metric graphs, tends towards \(\infty\) when the length of one edge shrinks to \(0\). This suggests that our spectral estimate \eqref{eq:intro-estimate-planar} is rather rough when the graph's edge lengths vary strongly. However, in Section \ref{subsec:EigMetr} we will discuss a number of examples  to show that, in fact, \eqref{eq:intro-estimate-planar} qualitatively improves known spectral estimates for some classes of metric graphs if the edge lengths are bounded from below.\\

In order to prove estimates \eqref{eq:intro-estimate-genus1}, \eqref{eq:intro-estimate-genus2} and \eqref{eq:intro-estimate-planar}, we make use of a spectral correspondence between the Kirchhoff Laplacian and a particular weighted combinatorial Laplacian to reduce the spectral problem to the combinatorial case. This correspondence has been studied very extensively by Exner, Kostenko, Malamud and Neidhardt \cite{ExnerKostenkoMalamudNeidhard:SpecTheo} in the case of infinite graphs. Although, for non-equilateral graphs, there this no explicit formula of the form \eqref{eq:VonBelow} relating the eigenvalues of \(\Delta_\mathcal G\) and this combinatorial Laplacian, there are estimates comparing their eigenvalues. Such an estimate was recently discovered by Kostenko and Nicolussi in \cite{KostenkoNicolussi:SpecEst} for the bottom of the spectrum of the Kirchhoff Laplacian on infinite graphs and, in fact, we will improve and extend their result in Section~\ref{subsec:spect-corr}. This will allow us to reduce the spectral estimates \eqref{eq:intro-estimate-genus1}, \eqref{eq:intro-estimate-genus2} and \eqref{eq:intro-estimate-planar} to spectral estimates for weighted combinatorial Laplacians.\\

	The combinatorial setting we are going to consider is the following:~given a finite graph \(G=(V,E)\), the weighted combinatorial Laplacian \(\mathcal L_{m,\mu}\) associated with a vertex weight \(m\) and an edge weight \(\mu\) on \(G\) is the operator given by
		\[(\mathcal L_{m,\mu} f)(v)=\frac{1}{m(v)}\sum_{e=\{u,v\}\in E_v}\mu(e)(f(v)-f(u)),\quad v\in V,\]
	defined on the space of functions \(f:V\rightarrow \C\) -- see Section \ref{subsec:prel-weight-comb-graphs} for a short introduction to the weighted combinatorial Laplacian. We also refer to \cite{BauKelWoj15} for Cheeger-type estimates, \cite{LenSchSto18} for estimates in dependence of the graphs diameter and references given therein for further recent research results on spectral estimates for weighted combinatorial Laplacians. If \(G\) is planar, we show in Theorem~\ref{thm:SpecBoundPlanComb} that the spectral gap \(\lambda_2(\mathcal L_{m,\mu})\) satifies the upper bound
	\begin{equation}\label{eq:intro-estimate-discrete1}
		\lambda_2(\mathcal{L}_{m,\mu})\leq 8\frac{\Deg{\mu}{\max}}{m(V)}
	\end{equation}
provided \(m\) does not concentrate too strongly in small regions of \(G\); more precisely, we have to assume that
	\begin{equation}\label{eq:2uvV}
		2\,(m(u)+m(v))<m(V)
	\end{equation}
holds for any pair of adjacent vertices \(u,v\in V\) -- here \(m(V)=\sum_{v\in V}m(v)\) denotes the total measure of \(G\) and \(\Deg{\mu}{\max}=\max_{v\in V}\sum_{e\in E_v}\mu(e)\) denotes the maximum degree of the vertices of \(G\) with respect to the weight \(\mu\). The condition~\eqref{eq:2uvV}, which we regard as a kind of smoothness of the weight function $m$, seems to be new in the literature. We will see that a spectral bound \eqref{eq:SpecBoundPlanComb} cannot hold for general vertex weights (see Remark~\ref{ex:CombExTeckAss}). However, in the special unweighted case \(m\equiv 1\) and \(\mu\equiv 1\) -- where the estimate \eqref{eq:intro-estimate-discrete1} coincides with the estimate \eqref{eq:intro-spielman-teng} by Spielman and Teng~\cite[Theorem 3.3]{SpielmanTeng:SpecPartWork07} -- the condition \eqref{eq:2uvV} does not need to be imposed, as we will discuss in Remark \ref{ex:CombExTeckAss}. The proof of \eqref{eq:intro-estimate-discrete1}, which is strongly inspired by the techniques used in \cite{SpielmanTeng:SpecPartWork07}, uses a representation for planar graphs via so-called circle packings (see Section 3.1) on the unit sphere in \(\mathbb R^3\). The main idea in the proof is to deform this circle-packing using a conformal map on the sphere to construct a sufficiently good test function in the Courant--Fischer Theorem for \(\lambda_2(\mathcal L_{m,\mu})\). During this deformation process our main task is going to be the identification of the condition \eqref{eq:2uvV} as a geometric condition for the existence of a certain type of circle packing that represents the weighted graph structure of \(G\) with the vertex weight \(m\) (see Lemma \ref{lem:Deformation}).\\

In the higher genus case -- in Section 4 -- we consider a special version of the weighted Laplacian \(\mathcal L_{m,\mu}\): for an edge weight \(\omega\) on \(G\), we consider the weighted normalized Laplacian on \(G\) given by
		\[(\mathcal L_{\mathrm{norm}}^\omega f)(v)=\frac{1}{\Deg{\omega}{v}}\sum_{e=\{u,v\}\in E_v}\omega(e)(f(v)-f(u)),\quad v\in V,\]
-- where \(\Deg{\omega}{v}=\sum_{e\in E_v}\omega\) is the weighted degree of the vertex \(v\) with respect to the edge weight \(\omega\) -- and show in Theorem~\ref{thm:est-normlapl} that its ordered eigenvalues \(0=\lambda_1(\mathcal L_{\mathrm{norm}}^\omega)<\lambda_2(\mathcal L_{\mathrm{norm}}^\omega)\leq\lambda_{|V_G|}(\mathcal L_{\mathrm{norm}}^\omega)\) satisfy
	\begin{equation}\label{eq:intro-estimate-discrete2}
		\lambda_k(\mathcal L_{\mathrm{norm}}^\omega)\leq C\frac{\Deg{\omega}{\max} (g+k)}{\sum_{e\in E} \omega(e)},\quad k=1,\ldots,|V_G|
	\end{equation}
with the maximal degree \(\Deg{\omega}{\max}=\max_{v\in V}\Deg{\omega}{v}\) and a generic constant \(C>0\). In the proof of this estimate,  we construct -- inspired by the techniques used by Amini and Cohen-Steiner in \cite{AminiCohenSteiner:TransferPrinc}, and extending them to the weighted case -- a Riemannian metric on a closed and oriented surface of genus \(g\) and a certain topological double cover of the surface, so that the \emph{vicinity graph} associated with this double cover (see Section \ref{subsec:ACS-theorem} for details) reflects not only the combinatorial structure of \(G\) but also the weighted structure induced by \(\omega\). From there we can exploit a transfer principle that was introduced in \cite{AminiCohenSteiner:TransferPrinc} to reduce the estimate \eqref{eq:intro-estimate-discrete2} to the estimate \eqref{eq:Hassannezhad} by Hassannezhad. The main difference in the construction compared with the one by Amini and Cohen-Steiner lies in the choice of the Riemannian metric, as we need to make sure that the Riemannian measure of the single double covering elements -- which are chosen so that they are isomorphic to compositions of Euclidean triangles -- correspond of the weight \(\omega\) while still preserving a certain convexity condition of the covering elements -- see Section \ref{subsec:proof-genus} for this construction.
\section{Preliminaries on combinatorial and metric graphs}
\subsection{Compact metric graphs}\label{subsec:prel-metric-graphs}
A \emph{compact metric graph} \(\mathcal G\) is the object obtained after gluing the end points of a finite number of bounded intervals in a graph-like way -- see \cite{mugnolo2019actually} for a rigorous definition. Usually, the set of intervals is referred to as the edge set of \(\mathcal G\) whereas the set of glued end points is referred to as the vertex set of \(\mathcal G\). As this gluing process naturally defines a combinatorial graph structure, we shall write \(\mathcal G=(G,\ell)\), where \(G=(V_G,E_G)\) is a connected combinatorial graph with finite vertex set \(V=V_G\) and finite edge set \(E=E_G\) and \(\ell:E\rightarrow(0,\infty),\ e\mapsto\ell_e\) is a length function that assigns the length \(\ell_e\) to a given edge \(e\in E\). We point out that this representation of \(\mathcal G\) is not unique since adding vertices on the edges changes the underlying combinatorial graph whereas the metric graph remains the same. Throughout this paper we assume that \(\mathcal G\) is connected as a topological space since the study of disconnected graphs is usually reduced to the study of its connected components.  For notational purposes we shall also fix an orientation on the combinatorial graph \(G\) -- let us however emphasize that the analytic and spectral properties of the objects considered in this paper do not depend on the choice of this particular orientation. For a given edge \(e\in E\) we write \(e_{\init}\) for its \emph{initial vertex} and \(e_\term\) for its \emph{terminal vertex} and for each vertex \(v\in V\) let \(E_{\init}^v\) and \(E_{\term}^v\) be the sets of edges initiating from \(v\) and terminating in \(v\) respectively. Then the \emph{metric degree} \(\Deg{\ell}{v}\) of a vertex \(v\in V\) is given by
	\[\Deg{\ell}{v}:=\sum_{e\in E_{\init}^v}\ell_e+\sum_{e\in E_{\term}^v}\ell_e.\]
(Note that \(e_\init=e_\term\) holds if and only if \(e\) is a loop, so loops incident to \(v\) are counted twice in the sum above.) The Hilbert space of square-integrable functions on \(\mathcal G\) is
	\[L^2(\mathcal G):=\bigoplus_{e\in E}L^2(0,\ell_e)\]
equipped with the scalar product given by
	\[(\varphi,\psi)_{L^2(\mathcal G)}:=\sum_{e\in E}\int_0^{l_e}\varphi_e(x_e) \bar\psi_e(x_e)\diff x_e\]
for \(\varphi=(\varphi_e)_{e\in E}\) and \(\psi=(\psi_e)_{e\in E}\) in \(L^2(\mathcal G)\).  The \emph{Kirchhoff Laplacian} \(\Delta_{\mathcal G}\) in \(L^2(\mathcal{G})\) is the operator acting edgewise as the negative second derivative, i.e.
	\[\Delta_{\mathcal G} \varphi=\left(-\varphi''_e\right)_{e\in E}\]
defined on the space of functions \(\varphi = (\varphi_e)_{e\in E}\), so that \(\varphi_e\in H^2(0,l_e)\) for all \(e\in E\) and the following two conditions are satisfied for every vertex \(v\in V\):
\begin{equation}
\begin{cases}
	\varphi \text{ is continuous at }v,\\
	\sum_{e\in E_{\term}^v}\varphi'_e(\ell_e)-\sum_{e\in E_{\init}^v}\varphi'_e(0)=0.
\end{cases}
\end{equation}
The Kirchhoff Laplacian is a nonnegative self-adjoint operator in \(L^2(\mathcal G)\) with discrete spectrum. The constant functions lie in the null space of $\Delta_{\mathcal G}$ and, as \(\mathcal G\) is connected, \(0\) is an eigenvalue of \(\Delta_{\mathcal G}\) of multiplicity \(1\). For a proof of these facts and further information on the analytic properties of the Kirchhoff Laplacian we refer to the text books \cite{BerkolaikoKuchment:IntrQG,Mugnolo:SemiMeth}. We shall denote the ordered eigenvalues of \(\Delta_{\mathcal G}\), counted with multiplicities, by
	\[0=\lambda_1(\Delta_{\mathcal G})<\lambda_2(\Delta_{\mathcal G})\leq \lambda_3(\Delta_{\mathcal G})\leq\ldots \rightarrow \infty.\]
The quadratic form \(Q_\mathcal G\) associated with \(\Delta_\mathcal G\) is given by
	\[Q_\mathcal G(\varphi)=\sum_{e\in E}\int_0^{\ell_e}|\varphi_e'(x_e)|^2\diff x_e\]
and its domain is the space consisting of all functions \(\varphi=(\varphi_e)_{e\in E}\) with \(\varphi_e\in H^1(0,\ell_e)\) for all \(e\in E\) that are continuous in every vertex \(v\in V\).
\subsection{Weighted combinatorial graphs}\label{subsec:prel-weight-comb-graphs}
Let \(G=(V,E)\) be a connected graph with finite edge and vertex sets. In the combinatorial setting we shall assume that the graph \(G\) is simple. This assumption simplifies some notation and is actually no restriction in the case of metric graphs, since -- as we mentioned before -- we may always add ``dummy'' vertices on loops or parallel edges to obtain a simple graph, which does not change the spectral properties of the Kirchhoff Laplacian (see \cite[Remark 1.4.2]{BerkolaikoKuchment:IntrQG}). We may thus identify each edge \(e\in E\) with the two element set \(\{u,v\}\) of its incident vertices \(u,v\in V\). Suppose \(m:V\rightarrow (0,\infty)\) is a positive weight functions on the vertex set. We think of \((V,m)\) as a finite measure space setting \(m(U):=\sum_{u\in U}m(u)\) for subsets \(U\subset V\). Let \(\ell^2_m(V)\) denote the vector space of complex valued functions \(f:V\rightarrow\C\) equipped with the weighted scalar product given by
		\[(f,g)_{\ell^2_m(V)}=\sum_{v\in V}m(v){f(v)}\overline{g(v)}.\]
	and its induced norm given by
	\begin{align}
		\|f\|_{\ell^2_m(V)}^2=\sum_{v\in V}m(v)|f(v)|^2\label{l2mnorm}
	\end{align}
	Moreover, let \(\mu:E\rightarrow (0,\infty)\) be a positive weight function on the edge set. The weighted degree of a vertex \(v\in V\) with respect to \(\mu\) is \(\Deg{\mu}{v}=\sum_{e\in E_v}\mu(e)\) and we set \(\Deg{\mu}{\max}:=\max_{v\in V}\Deg{\mu}{v}\). On \(\ell^2_m(V)\) we consider the nonnegative, self-adjoint operator \(\mathcal{L}_{m,\mu}:\ell^2_m(V)\rightarrow \ell^2_m(V)\) given by
	\begin{align}
		(\mathcal{L}_{m,\mu}f)(u)=\frac{1}{m(u)}\sum_{e=\{u,v\}\in E_v}\mu(e)(f(u)-f(v))\label{eq:CombLapl}
	\end{align}
for \(u\in V\). 
	We refer to \(\mathcal{L}_{m,\mu}\) as the weighted combinatorial Laplacian. Its associated quadratic form \(q:\ell^2_m(V)\rightarrow [0,\infty)\) is given by
	\begin{align}
		q(f)=\sum_{e=\{u,v\}\in E}\mu(e)|f(u)-f(v)|^2.\label{QuadraticFormq}
	\end{align}
	The ordered eigenvalues of \(\mathcal L_{m,\mu}\), counted with multiplicities, shall be denoted by
		\[0=\lambda_1(\mathcal L_{m,\mu})<\lambda_2(\mathcal L_{m,\mu})\leq \lambda_3(\mathcal L_{m,\mu})\leq\ldots\leq \lambda_{|V|}(\mathcal L_{m,\mu}).\]
	(Note that \(0\) is an eigenvalue of \(\mathcal L\) of multiplicity \(1\), as \(G\) is connected.) It will also be convenient to consider vector-valued functions. For \(d\in \N\) let \(\ell^2_m(V;\C^d)\) be the space of vector valued functions \(f:V\rightarrow \C^d\). On \(\ell^2_m(V;\C^d)\) we consider the norm \(\|\cdot\|_{\ell^2_m(V,\C^d)}\) and the quadratic form \(q\) given by the expressions in (\ref{l2mnorm}) and (\ref{QuadraticFormq}), where we replace the absolute value on \(\C\) with the Euclidean norm on \(\C^d\) respectively. For the spectral gap \(\lambda_2(\mathcal L_{m,\mu})\) of \(\mathcal L_{m,\mu}\) we have the formula
	\begin{align}
	\lambda_2(\mathcal{L}_{m,\mu})=\inf\left\{\frac{q(f)}{||f||_{\ell^2_m(V;\C^d)}^2}~\Big|~f\in \ell^2_m(V;\C^d)\setminus\{0\},~ \sum_{v\in V}m(v)f(v)=0\right\}.\label{eq:CombCourantFischer}
	\end{align}
	For \(d=1\) this is the classical Courant--Fischer principle for the spectral gap of \(\mathcal L_{m,\mu}\). However, Spielman and Teng \cite[Lemma 3.1]{SpielmanTeng:SpecPartWork07} proved that this variational principle can be extended to vector-valued test functions in the unweighted case (\(m\equiv 1,~\mu\equiv 1\)) and their proof can easily be generalized to the weighted case.
\subsection{Spectral correspondence between metric and combinatorial graphs}\label{subsec:spect-corr}
For simplicity of notation we will again restrict ourselves to simple graphs in this section. As we mentioned in the introduction, the proof of our main results in this paper make use of the correspondence between Kirchhoff and a related combinatorial Laplacian. If the metric graph \(\mathcal{G}\) is equilateral (\(l\equiv 1\)), this duality is well-known \cite{vonBelow:CharEqAs,Cattaneo:SpectContLapl}, where the vertex and edge weights of the related Laplacian -- the unweighted normalized Laplacian -- are given by \(m(v)=d_v=|E_v|\) and \(\mu(e)=1\). For general connected and compact metric graphs we consider the weighted combinatorial Laplacian \(\mathcal L_{m,\mu}\) with the weight given by
	\begin{align}
		m(v) & = \Deg{\ell}{v}=\sum_{e\in E_v}l_e, &
		\mu(e) & = \frac{1}{l_e}.\label{eq:InducedWeights}
	\end{align}
We can refer to \cite{ExnerKostenkoMalamudNeidhard:SpecTheo, KostenkoNicolussi:SpecEst} for further details on the common analytical and spectral properties of \(\Delta_\mathcal G\) and \(\mathcal L_{m,\mu}\). Our aim is to derive an optimal estimate that compares the eigenvalues of these two operators. Our proof is based on the following abstract result: it is probably already known, but we could not find an appropriate reference for it.
\begin{proposition}\label{prop:abstrhilb}
	Suppose \(H_1, H_2\) are two Hilbert spaces and \(J:H_1\hookrightarrow H_2\) is an injective linear bounded operator. For \(j=1,2\) let \(A_j\) be a lower semibounded self-adjoint operator on \(H_j\) with associated quadratic form \(q_j\) and suppose that \(A_2\) has discrete spectrum. Furthermore, we assume that there are constants \(\alpha,\beta>0\), so that
	\begin{enumerate}[label = (\roman*)]
	\item \(\|J x\|_{H_2}^2\geq \alpha \|x\|_{H_1}^2\) for all \(x\in H_1\),
	\item \(Jx\in D(q_2)\) and \(q_2(J x)\leq \beta q_1(x)\) for all \(x\in D(q_1)\).
	\end{enumerate}
	Then \(A_1\) has discrete spectrum and we have
	\begin{equation}\label{eq:abstract-EV-estimate}
		\lambda_k(A_2)\leq \frac{\beta}{\alpha}\lambda_k(A_1),\quad k=1,2,\ldots
	\end{equation}
	for the ordered eigenvalues \(\lambda_1(A_j)\leq \lambda_2(A_j)\leq\ldots\), counting multiplicities, of \(A_j\) for \(j=1,2\) respectively.	
\end{proposition}
\begin{proof}
	After restricting ourselves to the closed subspace \(\tilde H_2:=J(H_1)\) and the self-adjoint operator \(\tilde A_2\) associated with the quadratic form \(\tilde q_2:={q_2}_{|D(q_2)\cap \tilde H_2}\) on \(\tilde H_2\) we may assume that \(J\) is surjective and, thus, an isomorphism by (i). (Note that passing to \(\tilde H_2\) and \(\tilde A_2\) preserves the discreteness of the spectrum and only increases the eigenvalues of \(A_2\).)
	
	Then, by (i) and (ii),
	\begin{center}
	\begin{tikzcd}
	(D(q_1),\|\cdot\|_{q_1})\arrow[r,hook,"J_{|D(q_1)}"]\arrow[hook,d,"I_1"] & (D(q_2),\|\cdot\|_{q_2})\arrow[hook,d,"I_2"]\\
	(H_1,\|\cdot\|_{H_1})\arrow[r,"\sim","J"'] & (H_2,\|\cdot\|_{H_2})
	\end{tikzcd}	
	\end{center}
	is a commutative diagram of linear bounded operators, where \(||\cdot||_{q_j}\) denotes the norm induced by the quadratic form \(q_j\) as well as the scalar product on \(D(q_j)\) and \(I_j:D(q_j)\rightarrow H_j\) is the canonical embedding for \(j=1,2\) respectively. As \(A_2\) has discrete spectrum, \(I_2\) is a compact embedding. This implies -- using the commutativity of the diagram above and the fact that \(J\) is an isomorphism -- that \(I_1\) is compact, which in turn implies that \(A_1\) has discrete spectrum. Finally, the eigenvalue estimate \eqref{eq:abstract-EV-estimate} follows from the Courant--Fischer Theorem for the eigenvalues of \(A_1\) and \(A_2\) and the injectivity of \(J\).
\end{proof}
\begin{corollary}\label{cor:CompareSpectralGaps}
	Let \(\mathcal{G}=(G,\ell)\) be a connected, simple and compact metric graph. Let \(\mathcal{L}_{m,\mu}\) be the weighted combinatorial Laplacian on \(G\) with respect to the vertex weight \(m\) and the edge weight \(\mu\) defined in \eqref{eq:InducedWeights}. Then the eigenvalues of  \(\Delta_\mathcal G\) and \(\mathcal{L}_{m,\mu}\) satisfy the inequality
	\begin{align}
		\lambda_k(\Delta_\mathcal G)\leq \frac{\pi^2}{2}\lambda_k(\mathcal{L}_{m,\mu}),\quad k=1,\ldots,|V|.\label{eq:CompareSpectralGaps}
	\end{align}
\end{corollary}
\begin{proof}
	We consider the linear embedding \(J:\ell^2_m(V)\hookrightarrow L^2(\mathcal G)\) that assigns a given function \(f\in \ell^2_m(V)\) the edgewise trigonometric function \(Jf=(\varphi_e)_{e\in E}\in D(Q_\mathcal G)\) given by
		\[\varphi_e(x_e)=\frac{f(e_\init)+f(e_\term)}{2}+\frac{f(e_\init)-f(e_\term)}{2}\cos\left(\frac{x_e\pi}{\ell_e}\right)\]
	for \(e\in E\) and \(x_e\in [0,\ell_e]\). After deriving the respective \(L^2\)-norms of \(\varphi_e\) and \(\varphi_e'\) one obtains
	\begin{align*}
		\|Jf\|_{L^2(\mathcal G)}^2 & =\sum_{e\in E}\frac{\ell_e}{8}\left(3|f(e_\term)|^2+2\mathrm{Re}\left(f(e_\term)\overline{f(e_\init)}\right)+3|f(e_\init)|^2\right)\\
		& \geq \frac{1}{4}\sum_{e\in E}\ell_e\left(|f(e_\term)|^2+|f(e_\init)|^2\right) =\frac{1}{4}\sum_{v\in V}\Deg{\ell}{v}|f(v)|^2 =\frac{1}{4}\|f\|_{\ell^2_m(V)}
	\end{align*}
	and
	\begin{align*}
		Q_\mathcal G(Jf) =\frac{\pi^2}{2}\sum_{e\in E}\frac{1}{\ell_e}|f(e_\term)-f(e_\init)|^2=\frac{\pi^2}{8}q(f).
	\end{align*}
	We can hence apply Proposition~\ref{prop:abstrhilb} with the operators \(A_1=\mathcal{L}_{m,\mu}\) and \(A_2=\Delta_{\mathcal G}\), and applying \eqref{eq:abstract-EV-estimate} with \(\alpha=\frac{1}{4}\) and \(\beta=\frac{\pi^2}{8}\) proves the claim.
\end{proof}
\begin{remark}
	\begin{remarklist}
	\item Our estimate \eqref{eq:CompareSpectralGaps} is in line with von Below's formula \eqref{eq:VonBelow}. As \(\arccos(1-\lambda)^2\leq \frac{\pi^2}{2}\lambda\) holds for each \(\lambda\in [0,2]\), \eqref{eq:VonBelow} implies that \(\lambda_2(\Delta_\mathcal G)\leq \frac{\pi^2}{2}\lambda_2(\mathcal L_{\mathrm{norm}})\) if \(\mathcal G\) is equilateral.
	 \item A slightly weaker version of \eqref{eq:CompareSpectralGaps} was recently proved by Kostenko and Nicolussi \cite[Lemma 2.10]{KostenkoNicolussi:SpecEst} in the setting of infinite metric graphs:  they have showed that
\begin{align*}	 
	 \inf\sigma(\Delta_\mathcal G^F)  \leq 6\inf\sigma(\mathcal L_{m,\mu}^F)\quad\hbox{and}\quad \inf\sigma_{\mathrm{ess}}(\Delta_\mathcal G^F)  \leq 6\inf\sigma_{\mathrm{ess}}(\mathcal L_{m,\mu}^F)
\end{align*}
hold for the bottom of the (essential) spectra of the Friedrichs extensions \(\Delta_\mathcal G^F\) and \(\mathcal L_{m,\mu}^F\) of compactly supported versions of the Kirchhoff Laplacian and discrete Laplacian respectively. To obtain this result they have chosen edgewise linear functions in the Courant--Fischer Theorem, rather than edgewise trigonometric functions. Our proof also applies to this case. Indeed, it can be shown that the estimates
\begin{align*}	
	\inf\sigma(\Delta_\mathcal G^F) \leq \frac{\pi^2}{2}\inf\sigma(\mathcal L_{m,\mu}^F)\quad\hbox{and}\quad \inf\sigma_{\mathrm{ess}}(\Delta_\mathcal G^F) \leq \frac{\pi^2}{2}\inf\sigma_{\mathrm{ess}}(\mathcal L_{m,\mu}^F)
\end{align*}
hold. Furthermore, the statement of Corollary \ref{cor:CompareSpectralGaps} also holds for the eigenvalues of \(\Delta_\mathcal G^F\)  and \(\mathcal L_{m,\mu}^F\) on infinite graphs if  \(\Delta_\mathcal G^F\) has discrete spectrum.
	\item The estimate \eqref{eq:CompareSpectralGaps} is in fact sharp for all $k$. In order to see, we this consider the equilateral star graph  \(\mathcal S_n\) with \(n\) edges of constant length \(\ell\equiv 1\). The smallest \(n+1\) eigenvalues of \(\Delta_{\mathcal S_n}\) are \(0\), \(\frac{\pi^2}{4}\) (of multiplicity \(n-1\)) and \(\pi^2\). The eigenvalues of the corresponding normalized Laplacian are \(0\), \(1\) (of multiplicity \(n-1\)) and \(2\). Thus, equality is achieved in \eqref{eq:CompareSpectralGaps} for \(k=n+1\).
	\end{remarklist}
\end{remark}
\section{Eigenvalue bounds for Planar Graphs}\label{sec:eigenvboundplanar}
\subsection{A Technical Tool: Circle Packings for Planar Graphs}\label{subsec:CirclePack}
We recall that a finite graph \(G=(V,E)\) is called \emph{planar} if there exists a drawing of \(G\) in the plane, such that every edge in \(E\) is represented by a Jordan curve and any two of these Jordan curves only intersect at their respective endpoints if the associated edges are incident. Although this classical definiton of planarity is based on topological concepts, a purely combinatorial characterization of planarity is available: Kuratowski's Theorem \cite{Kuratowski:Sur} states that a finite graph \(G=(V,E)\) is planar if and only if a subgraph of \(G\) is a subdivision of the complete graph \(K_5\) or of the complete bipartite graph \(K_{3,3}\). Here, a subdivsion of some graph \(G\) is a graph obtained after successively inserting vertices on the edges of \(G\). A different, more geometric concept to characterize planar graphs is stated in the Circle Packing Theorem by Koebe, Andreev and Thurston \cite{Koebe:Kontaktpropleme, Andreev:convpoly1, Andreev:convpoly2, Thurston:TheGeoAndTop}:
\begin{theorem}
	A simple, finite graph \(G\) is planar if and only if there exists a family \((D_v)_{v\in V}\) of closed disks \(D_v\subset \R^2\), such that the following holds for any two vertices \(v\neq u\) in \(V\):
	\begin{enumerate}[label = (\roman*)]
		\item If \(v\) and \(u\) are adjacent, the disks \(D_v\) und \(D_u\) intersect in exactly one point.
		\item If \(v\) and \(u\) are not adjacent, the disks \(D_v\) and \(D_u\) are disjoint.
	\end{enumerate}
	In this case, we call \((D_v)_{v\in V}\) a univalent circle packing of \(G\) in the plane.
\end{theorem}
	Let us transfer the concept of circle packings to the unit sphere:
	\begin{definition}	
	A subset \(k\) of the unit sphere \(S^2\subset \R^3\) is called a circular line, if it is the non-trivial intersection of \(S^2\) with a hyperplane \(H\subset\R^3\). (By non-trivial we mean that \(k\) is neither empty nor just one point). A connected, closed subset \(C\subset S^2\) is called a spherical cap, if its boundary in \(S^2\) is a circular line.
	\end{definition}
\begin{figure}
\centering
\begin{minipage}{0.5\textwidth}
	\tikzstyle{my style}=[circle, draw, fill=black!50,
                        inner sep=0pt, minimum width=4pt]
        \centering
        \begin{tikzpicture}[scale=0.5]%
        \draw[color=gray] (0,0) circle [radius=(3-1.5*sqrt(3))*1cm] ;
    	\draw \foreach \x in {90,210,330} {
      (\x:3cm) node[my style]{}-- (\x+120:3cm)
      (\x:3cm) node[circle through = {($(\x:3cm)!0.5!(\x+120:3cm)$)},draw=gray]{}
      (\x:3)--(0,0)
      (0,0) node[my style]{}
};
\end{tikzpicture}
        \caption{A univalent circle packing for the complete graph \(K_4\)}
\end{minipage}\hfill
\begin{minipage}{0.5\textwidth}
\pgfplotsset{compat=1.10}

\pgfdeclareradialshading[tikz@ball]{ball}{\pgfqpoint{0bp}{0bp}}{%
 color(0bp)=(tikz@ball!0!white);
 color(7bp)=(tikz@ball!0!white);
 color(15bp)=(tikz@ball!70!black);
 color(20bp)=(black!70);
 color(30bp)=(black!70)}
\makeatother

\tikzset{viewport/.style 2 args={
    x={({cos(-#1)*1cm},{sin(-#1)*sin(#2)*1cm})},
    y={({-sin(-#1)*1cm},{cos(-#1)*sin(#2)*1cm})},
    z={(0,{cos(#2)*1cm})}
}}

\pgfplotsset{only foreground/.style={
    restrict expr to domain={rawx*\CameraX + rawy*\CameraY + rawz*\CameraZ}{-0.05:100},
}}
\pgfplotsset{only background/.style={
    restrict expr to domain={rawx*\CameraX + rawy*\CameraY + rawz*\CameraZ}{-100:0.05}
}}

\def\addFGBGplot[#1]#2;{
    \addplot3[#1,only foreground] #2;
}

\newcommand{\ViewAzimuth}{-20}
\newcommand{\ViewElevation}{30}
\centering
\vspace{-0.5cm}
\begin{tikzpicture}
    \pgfmathsetmacro{\CameraX}{sin(\ViewAzimuth)*cos(\ViewElevation)}
    \pgfmathsetmacro{\CameraY}{-cos(\ViewAzimuth)*cos(\ViewElevation)}
    \pgfmathsetmacro{\CameraZ}{sin(\ViewElevation)}
    \pgfmathsetmacro{\Radius}{2.5}
    \begin{scope}
        \clip (0,0) circle (\Radius*1cm);
        \begin{scope}[transform canvas={rotate=-20}]
            \shade [ball color=white] ({0,0.5*\Radius}) ellipse ({1.8*\Radius} and {1.5*\Radius});
        \end{scope}
    \end{scope}
    \begin{axis}[
        hide axis,
        view={\ViewAzimuth}{\ViewElevation},     
        every axis plot/.style={very thin},
        disabledatascaling,                      
        anchor=origin,                           
        viewport={\ViewAzimuth}{\ViewElevation}, 
    ]

        \addFGBGplot[domain=0:2*pi, samples=51, samples y=11,smooth,
        domain y=40:90,surf,shader=flat,color=blue!30,opacity=1] 
            ({\Radius*cos(deg(x))*cos(y)},
            {\Radius*sin(deg(x))*cos(y)}, {\Radius*sin(y)});
		\addFGBGplot[domain=0:2*pi, samples=100, samples y=1] ({\Radius*cos(deg(x))}, {\Radius*sin(deg(x))}, 0);
        \addFGBGplot[domain=0:2*pi, samples=100, samples y=1] (0, {\Radius*sin(deg(x))}, {\Radius*cos(deg(x))});
        \addFGBGplot[domain=0:2*pi, samples=100, samples y=1] ({\Radius*sin(deg(x))}, 0, {\Radius*cos(deg(x))});
        
        \pgfmathsetmacro{\Boundx}{-145}
        \pgfmathsetmacro{\Boundy}{40}
        
        \addFGBGplot[domain=0:2*pi, samples=100, color=blue!80!black, samples y=1] 
            ({\Radius*cos(deg(x))*cos(\Boundy)},
            {\Radius*sin(deg(x))*cos(\Boundy)}, {\Radius*sin(\Boundy)});
        \draw[color=red!50!black,dashed] (0,0,\Radius) --     node[right]{\tiny{\(r(C)\)}} ({\Radius*cos(\Boundx)*cos(\Boundy)},
            {\Radius*sin(\Boundx)*cos(\Boundy)}, {\Radius*sin(\Boundy)});
         \draw (0,0,\Radius) node[circle, fill=black,
                        inner sep=0pt, minimum width=4pt]{};
         \draw ({\Radius*cos(-45)*cos(75)},
            {\Radius*sin(-45)*cos(75)}, {\Radius*sin(75)}) node{\tiny{\(p(C)\)}};
         \draw[color=blue!80!black] ({\Radius*cos(-200)*cos(50)},
            {\Radius*sin(-200)*cos(50)}, {\Radius*sin(50)}) node{\tiny{\(C\)}};
         \draw[color=blue!80!black] ({\Radius*cos(-60)*cos(36)},
            {\Radius*sin(-60)*cos(36)}, {\Radius*sin(36)}) node{\tiny{\(k\)}};
\end{axis}
\end{tikzpicture}
\caption{A spherical cap on \(S^2\).}
\end{minipage}\hfill
\end{figure}%
\begin{remark}
If \(C\) is a spherical cap bounded by a circular line \(k\), there exists exactly one point \(p(C)\in C\) with equal Euclidean distance to any point in \(k\). We call \(p(C)\) the \emph{center} and \(r(C)\) the \emph{radius} of \(C\). Let us emphasize that \(r(C)\) is the Euclidean distance in \(\R^3\), not the geodesic distance on the sphere. The surface area of \(C\) is given by \(\pi\cdot r(C)^2\).
\end{remark}
	A circle packing in \(S^2\) is a familiy \(\mathcal{C}=(C_v)_{v\in V}\) of spherical caps in \(C_v\subset S^2\) over some finite index set \(V\). It is called univalent if the interiors of the caps in \(\mathcal{C}\) are mutually disjoint. Given a circle packing \(\mathcal{C}=(C_v)_{v\in V}\) in \(S^2\) we define its intersection graph as the simple graph \(G=(V,E)\) with edge set
	\[E=\big\{\{u,v\}~|~u,v\in V, u\neq v\text{ and } C_u\cap C_v\neq\emptyset\big\}.\]

The stereographic projection maps disks in the plane to spherical caps in the sphere and vice versa (see Section \ref{sec:CircPresMaps}). Therefore, the Koebe--Andreev--Thurston Theorem can be reformulated by means of circle packings in the sphere in the following way:
\begin{corollary}\label{cor:KoebeAndreevThurston}
	 A simple, finite graph is planar if and only if it is the intersection graph of a univalent circle packing in the unit sphere.
\end{corollary}
This reformulation of the Circle Packing Theorem is a strong tool in the construction of good separators for planar graphs (see \cite{Miller:Separators, Miller:FiniteElMeshes}). In our case we will use the circle packing representation to construct a test function in the Courant--Fischer--Theorem \eqref{eq:CombCourantFischer}. Note that the circle packing \(\mathcal C\) in Corollary \ref{cor:KoebeAndreevThurston} is certainly not unique, since any bijective and conformal map \(f:S^2\rightarrow S^2\) maps \(\mathcal C\) to a different circle packing with the same intersection graph \(G\). We will benefit from this non-uniqueness, as it enables us to adjust the circle packing to the vertex weight \(m\). The main task of this subsection will be the proof of the following technical lemma:
\begin{lemma}\label{lem:Deformation}
Let \(\mathcal{C}=(C_v)_{v\in V}\) be a univalent circle packing in \(S^2\) and let \ \(m:V\rightarrow (0,\infty)\) be a positive weight function on \(V\) with
	\begin{align}		
		2\left(m(u)+m(v)\right)<m(V) \text{ for all } u,v\in V\text{ with }C_u\cap C_v\neq\emptyset.\label{eq:TechAssumpDef}
	\end{align}
	Then there exists a homeomorphism \(f:S^2\rightarrow S^2\), that maps spherical caps to spherical caps, such that the (univalent) image circle packing \(\tilde{\mathcal{C}}=\left(f(C_v)\right)_{v\in V}\) satisfies
	\begin{align}
		\sum_{v\in V}m(v)p(f(C_v))=0.
	\end{align}\label{eq:def-m-admissible}
\end{lemma}
\subsubsection{Circle-preserving maps and the proof of Lemma \ref{lem:Deformation}}\label{sec:CircPresMaps}
	The aim of this section is to prove Lemma \ref{lem:Deformation}. We begin by indroducing some notation: for \(\beta\in S^2\) let \(H_\beta\) denote the affine hyperplane in \(\R^3\) tangential to the unit sphere \(S^2\) at the point \(\beta\), that is
	\[H_\beta:=\{y\in\R^3~|~ \langle y-\beta,\beta\rangle=0\}.\]
Furthermore, let \({\pi_\beta}:{H_\beta}\rightarrow{S^{2}\setminus\{-\beta\}}\) be the stereographic projektion of \(H_\beta\) onto \(S^{2}\), that is
	\begin{align*}
		\pi_\beta(y) & = \frac{4}{|\beta+y|^2}(\beta+y)-\beta
	\intertext{for \(y\in H_\beta\) and}
		\pi_\beta^{-1}(z) & = \frac{1}{1+\langle z,\beta\rangle}(\beta+z)-\beta
	\end{align*}
	for \(z\in S^2\setminus\{-\beta\}\). Here \(\langle\cdot,\cdot\rangle\) denotes the Euclidean scalar product on \(\R^3\). With the usual convention \(\pi_\beta(\infty)=-\beta\) the projection \(\pi_\beta\) extends to a homeomorphism \({\pi_\beta}:{H_\beta\cup\{\infty\}}\rightarrow{S^{2}}\). The stereographic projection is circle-preserving in the following sense:
\begin{itemize}
	 \item[(1a)] If \(k\) is a circular line in \(H_\beta\), then \(\pi_\beta(k)\) is a circular line in \(S^2\).
	 \item[(1b)] If \(g\) is a straight line in \(H_\beta\), then \(\pi_\beta(g)\cup\{-\beta\}\) is a circular line in \(S^2\).
	 \item[(2)] If \(k'\) is a circular line in \(S^2\), then either \(\pi_\beta(k')\) is a circular line in \(S^2\), if \(-\beta\notin k'\), or \(\pi_\beta(k')\setminus\{\infty\}\) straight line in \(H_\beta\), if  \(-\beta\in k'\).
\end{itemize}	
	A geometric proof of these facts can be found in \cite[\S 36]{HilbertCohnvossen:GeoIm}. Additionally, for \(\lambda>0\) we consider the dilation \(D_\beta^\lambda\) on \(H_\beta\) with centre \(\beta\) und factor \(\lambda\), i.e.\
		\[D_\beta^\lambda(y)=\beta+\lambda(y-\beta)\]
	for \(y\in H_\beta\). Again we extend \(D_\beta^\lambda\) to a homeomorphism \({D_\beta^\lambda}:{H_\beta\cup\{\infty\}}\rightarrow{H_\beta\cup\{\infty\}}\) via \(D_\beta^\lambda(\infty)=\infty\), that maps straight lines to straight lines and circular lines to circular lines. For arbitrary \(\lambda>0\) and \(\beta\in S^2\) we set
	\[{g_\beta^\lambda:=\pi_\beta^{-1}\circ D_\beta^\lambda\circ\pi_\beta}:{S^2}\rightarrow{S^2}.\] Then \(g_\beta^\lambda\) is a homeomorphism, that maps circular lines to circular lines. Moreover, the map
	\begin{align*}
		(0,\infty)\times S^2\times S^2 & \rightarrow S^2\\
		(\lambda,\beta,z) & \mapsto g_\beta^\lambda(z)
	\end{align*}
is continuous. Next, we shall study the behavior of \(g_\beta^\lambda\) as \(\lambda\rightarrow 0\):
\begin{itemize}
\item For \(\beta, z\in S^2\) with \(z\neq -\beta\) and for sequences \((\beta_n)_n\subset S^2, (z_n)_n\subset S^2\) and \((\lambda_n)_n\subset (0,\infty)\) with
	\[(\beta_n,\lambda_n,z_n)\rightarrow (\beta,0,z),\]
we have
	\[g_{\beta_n}^{\lambda_n}(z_n)\rightarrow \beta\]
as \(n\rightarrow\infty\).
\item Contrariwise, for \(z=-\beta\) we have \(g_\beta^\lambda(z)=-\beta\) for all \(\lambda>0\).
\end{itemize}
Therefore we define the limit map \({g_\beta^0}:{S^2}\rightarrow{S^2}\) for \(\beta\in S^2\) by
\begin{align*}
	g_\beta^0(z)=
	\begin{cases}
		-\beta, & \text{if } z=-\beta\\
		\beta, & \text{else}
	\end{cases}
\end{align*}
for \(z\in S^2\). With this extension the map
	\(
		[0,\infty)\times S^2\times S^2  \rightarrow S^2,~
		(\lambda,\beta,z)  \mapsto g_\beta^\lambda(z)
	\)
is continuous on the relatively open subset \([0,\infty)\times S^2\times S^2\setminus\{(0,\beta,-\beta)~|~\beta\in S^2\}\). An immediate conclusion of this continuity is the following
\begin{lemma}\label{lem:UniformConv}
	Let \(K,L\subset S^2\) be compact subsets, such that \(z\neq -\beta\) holds for all \((\beta,z)\in K\times L\), and let \(\varepsilon>0\). Then there exists a \(\delta\in (0,1)\) with
		\[|g_\beta^\lambda(z)-\beta|<\varepsilon\]
	for all \((\lambda,\beta,z)\in [0,\delta]\times K\times L\).
\end{lemma}
As we are primarily interested in circle packings in \(S^2\), our next step will be to study the behavior of \(g_\beta^\lambda\) on a fixed circular cap \(C\subset S^2\). First, note that \(g_\beta^\lambda(C)\) is a circular cap for \(\lambda>0\), since \(g_\beta^\lambda\) is a circle-preserving homeomorphism, and therefore the centre \(p\left(g_\beta^\lambda(C)\right)\) is well-defined. Because of the continuous dependence of \(g_\beta^\lambda\) on \((\lambda,\beta)\in (0,\infty)\times S^2\) the centre \(p\left(g_\beta^\lambda(C)\right)\) also depends continuously on \((\lambda,\beta)\). Moreover, applying Lemma \ref{lem:UniformConv} to \(K=C\) and \(L=\{\beta\}\) yields
	\(p\left(g_\beta^\lambda(C)\right)\rightarrow \beta\)
as \(\lambda\rightarrow 0\) for \(\beta\in S^2\) with \(-\beta\notin C\). That is why we define
\begin{align*}
	p\left(g_\beta^0(C)\right):=\begin{cases}
		\beta, & \text{if } -\beta\notin C,\\
		-\beta, & \text{if } -\beta\in C.
	\end{cases}
\end{align*}
With this convention the map
\(
	[0,\infty)\times S^2  \rightarrow S^2,~
	(\lambda,\beta) \mapsto p\left(g_\beta^\lambda(C)\right)
\)
is continuous on the set \([0,\infty)\times S^2\setminus (\{0\}\times -C)\). Let us finally consider the maps
	\[f_\alpha=\begin{cases}
		g_{{\alpha}/{|\alpha|}}^{1-|\alpha|}, & \text{if } \alpha\neq 0,\\
		\Id_{S^2}, & \text{if } \alpha=0.
	\end{cases}\]
	for \(\alpha\in B^3\). The corresponding result for \(g_\beta^\lambda\) yields that \(\alpha\mapsto p\left(f_\alpha(C)\right)\) is continuous on \(B^3\setminus (-C)\). Also note that
\begin{align}\label{eq:CapCentreAtBound}
	p\left(f_\alpha(C)\right)=\begin{cases}
		\alpha, & \text{if } -\alpha\notin C,\\
		-\alpha, & \text{if } -\alpha\in C
	\end{cases}
\end{align}
for \(\alpha\in S^2\). The main tool of the proof of Lemma \ref{lem:Deformation} is the following conclusion of the fixed-point theorem of Brouwer.
\begin{lemma}\label{lem:fixpoint}
	Let \({\Phi}:{B^3}\rightarrow{\R^3}\) be a continuous map and assume that for any \(\alpha\in S^2\) the image \(\Phi(\alpha)\) lies on the ray initiating at the origin and passing through \(\alpha\). Then there exists some \(\hat\alpha\in B^3\) with \(\Phi(\hat\alpha)=0\).
\end{lemma}
\begin{proof}
	If there were no such \(\hat\alpha\), the map \({-\frac{\Phi}{|\Phi|}}:{B^3}\rightarrow{B^3}\) would be well-defined and continuous, but because of our assumptions on \(\Phi\) it would not have any fixed point in \(B^3\) -- a contradiction to the fixed-point theorem of Brouwer.
\end{proof}
	Given a weight function \(m:V\rightarrow (0,\infty)\) and a univalent circle packing \(\mathcal{C}=(C_v)_{v\in V}\) we would like to apply this Lemma to the map \(\Phi\) given by
	\begin{align}\label{eq:DiscontFunction}
		\Phi(\alpha)=\sum_{v\in V}m(v)p(f_\alpha(C_v)), \quad \alpha\in B^3,
	\end{align}
	and, in fact, one easily checks that \(\Phi\) satisfies the ray condition required in the previous lemma provided that~\eqref{eq:TechAssumpDef} holds; but unfortunately \(\Phi\) will be discontinuous on \(-C_v\) for \(v\in V\). This is why we will smoothen \(\Phi\) close to the caps \(-C_v\) without changing \(\Phi\) too much in the interior of \(B^3\). The following lemma ensures that we have some kind of control of the behavior of the caps \(f_\alpha(C_v)\) for \(\alpha\) near \(S^2\). Roughly speaking, it states that 'most'  of the caps converge in some sense uniformly to \(\alpha/|\alpha|\) as \(\alpha\) gets closer to the boundary of \(B^3\). To classify what we mean by 'most' we shall define the system
		\[\mathcal{V}_\mathcal{C}:=\left\{U\subset V~|~\bigcap_{v\in U}C_v\neq\emptyset\right\}\]
Note that, since the circle packing is univalent, \(\mathcal{V}_{\mathcal{C}}\) only consist of the one element subsets of \(V\) and the two element subsets \(\{u,v\}\subset V\) with \(C_u\cap C_v\neq\emptyset\). In particular, \eqref{eq:TechAssumpDef} is equivalent to
	\begin{align}
		2\cdot m(U)<m(V) \text{ for all } U\in\mathcal{V}_\mathcal{C}.
	\end{align}
	\begin{lemma}\label{lem:CapsCloseToBoundary}
		Let \(\varepsilon\in(0,1)\). Then we find some \(\delta\in(0,1)\), so that for any \(\alpha\in B^3\) with \(1-\delta<|\alpha|<1\) we have
			\[V_\alpha^2:=V\setminus V_\alpha^1\in \mathcal{V}_\mathcal{C},\]
		where
			\[V_\alpha^1 :=\{v\in V~|~f_\alpha(C_v)\subset B_\varepsilon(\alpha/|\alpha|)\}.\]
	\end{lemma}
	
	\begin{proof}
	For arbitrary \(\beta\in S^2\) let \(W_\beta\subset V\) be the set of \(v\in V\) with \(-\beta\notin C_v\) and let \(L_\beta\) be the union of all the caps \(C_v\) with \(v\in W_\beta\). Note, that \(V\setminus W_\beta\in \mathcal{V}_\mathcal{C}\).  Since \(L_\beta\) is compact we find an open neighbourhood \(U_\beta\subset S^2\) of \(\beta\in U_\beta\) with \((-\overline{U_\beta})\cap L_\beta=\emptyset\). We Lemma \ref{lem:UniformConv} to \(L_\beta\) and \(K_\beta=\overline{U_\beta}\) to find some \(\delta_\beta\in (0,1)\), such that \(g_\gamma^\lambda(L_\beta)\subset B_\varepsilon(\gamma)\) for all \(\lambda\in (0,\delta_\beta)\) and all \(\gamma\in U_\beta\). In particular, \(f_\alpha(C_v)\subset B_\varepsilon(\alpha/|\alpha|)\) holds for \(v\in W_\beta\) and for all \(\alpha\in\R^3\) with \(1-\delta_\beta<|\alpha|<1\) and \(\alpha/|\alpha|\in U_\beta\).
	
	Using the compactness of \(S^2\) we find finitely many \(\beta_i\in S^2\), so that \(S^2\) is the union of the associated neighbourhoods \(U_{\beta_i}\). We set \(\delta=\min_i \delta_i\). Now, for arbitrary \(\alpha\in \R^3\) with \(1-\delta<|\alpha|<1\) there is some \(\beta_i\) with \(\alpha/|\alpha|\in U_{\beta_i}\). By construction
		\[W_{\beta_i}\subset \{v\in V~|~f_\alpha(C_v)\subset B_\varepsilon(\alpha/|\alpha|)\}=V_\alpha^1\]
	holds. Since \(V_\alpha^2=V\setminus V_\alpha^1\) is a subset \(V\setminus W_{\beta_i}\in\mathcal{V}_\mathcal{C}\), the set \(V_\alpha^2\) is also in the system \(\mathcal{V}_\mathcal{C}\).
\end{proof}
	\begin{proof}[Proof of Lemma \ref{lem:Deformation}]
	Let \((C_v)_{v\in V}\) be a circle packing and \(m:V\rightarrow (0,\infty)\) be a positve function. We assume that (\ref{eq:TechAssumpDef}) holds. As mentioned before, the main strategy of the proof is to insert a term in (\ref{eq:DiscontFunction}), that properly smoothens \(\Phi\) close to the boundary of \(B^3\), and then apply Lemma \ref{lem:fixpoint} to this new function.
	
	For (a yet to be chosen) \(\varepsilon\in(0,1)\) choose \(\delta\in (0,1)\) as in Lemma \ref{lem:CapsCloseToBoundary}. For \(v\in V\) we consider the maximum distance function given by
		\[\dist(\alpha,C_v)=\max_{z\in C_v}|\alpha-z| \text{ for } \alpha\in B^3.\]
	Moreover, we define the continuous function \({w_v}:{B^3}\rightarrow{[0,1]}\) given by
		\[w_v(\alpha)=\begin{cases}
			\frac{2-\dist(\alpha,C_v)}{\delta}, & \text{if } \dist(\alpha,C_v)\geq 2-\delta,\\
			1, & \text{else}.\end{cases}\]
	Obviously the distance function can be bounded using
	\begin{align}\label{eq:EstimateDistance}
		\dist(\alpha,C_v)\leq 1+|\alpha|
	\end{align}
	for any \(\alpha\in B^3\setminus\{0\}\) and, more precisely, equality holds in (\ref{eq:EstimateDistance}), if and only if \(-\alpha/|\alpha|\) is in \(C_v\). Thus \(w_v(\alpha)=0\) for \(\alpha\in B^3\), if and only if   \(-\alpha\in C_v\). In particular \(w_v\) vanishes in the discontinuity points of the bounded map \(\alpha\mapsto p\left(f_\alpha(C_v)\right)\). Therefore, the assignment
		\[\alpha\mapsto w_v(\alpha)p\left(f_\alpha(C_v)\right), ~\alpha\in B^3\]
	defines a continuous map on \(B^3\) for any \(v\in V\) and the map \({\Phi}:{B^3}\rightarrow{\R^3}\) given by
		\[\Phi(\alpha)=\sum_{v\in V}m(v)w_v(\alpha)p\left(f_\alpha(C_v)\right), ~\alpha\in B^3,\]
	is continuous. Using (\ref{eq:CapCentreAtBound}) we get
		\[\Phi(\alpha)=\sum_{\begin{subarray}{c}~v  \in V\\ -\alpha  \notin C_v\end{subarray}} m(v)w_v(\alpha)\alpha\]
	 for all \(\alpha\in S^2\), so \(\Phi(\alpha)\neq 0\) is on the ray starting in \(0\) and passing through \(\alpha\). Applying Lemma \ref{lem:fixpoint} we conclude that there exists some \(\alpha\in \R^3\) with \(|\alpha|<1\) and 
	 	\[0=\Phi(\alpha)=\sum_{v\in V}m(v)w_v(\alpha)p\left(f_{\alpha}(C_v)\right).\]
	 Now, we shall prove that \(\alpha\) satisfies \(|\alpha|\leq 1-\delta\) if we choose \(\varepsilon\) sufficiently small since -- if this holds -- we obtain
	\[\dist(\alpha,C_v)\leq 1+|\alpha|<2-\delta\]
and thus \(w_v(\alpha)=1\) for all \(v\in V\). Using (\ref{eq:EstimateDistance}) this would imply
	\[0=\sum_{v\in V}m(v)p\left(f_\alpha(C_v)\right),\]
so choosing \(f=f_\alpha\) would prove Theorem \ref{lem:Deformation}.

Let us contrarily assume that \(1-\delta<|\alpha|<1\) holds. Recall that \(V_\alpha^2=V\setminus V_\alpha^2\in V_\mathcal{C}\) by the choice of \(\delta\) in Lemma \ref{lem:CapsCloseToBoundary}, where 
\begin{align*}
	V_\alpha^1=\{v\in V~|~f_\alpha(C_v)\subset B_\varepsilon\left(\alpha/|\alpha|\right)\}.
\end{align*}
By definition of the  sets \(V_\alpha^1\) and \(V_\alpha^2\) we have \(\dist(\alpha,f_\alpha(C_v))\leq \dist(\alpha,f_\alpha(C_u)) \text{ for } v\in V_\alpha^1 \text{ and } u\in V_\alpha^2.\) Now, it can be seen that \(\dist(\alpha,C_v)\leq \dist(\alpha,C_u)\). This follows from the fact that \(f_\alpha\) corresponds to a dilation on the hyperplane tangential to \(S^2\) at \(\alpha/|\alpha|\). We thus obtain \(w_v(\alpha)\geq w_u(\alpha)\) for \(v\in V_\alpha^1\) and \(u\in V_\alpha^2\). Setting \(C_\alpha:=\min_{v\in V_\alpha^1}w_v(\alpha)>0,\) we obtain \(w_v(\alpha)\geq C_\alpha\geq w_u(\alpha)\) for \(v\in V_\alpha^1\) and \(u\in V_\alpha^2\). Note that
	\[\left|\frac{\alpha}{|\alpha|}-p\left(f_\alpha(C_v)\right)\right|\leq\varepsilon\]
	holds for all \(v\in V_\alpha^1\), so we may estimate
\begin{align*}
	\Big|\sum_{v\in V_\alpha^1}m(v)w_v(\alpha)p\left(f_\alpha(C_v)\right)\Big| &
	\geq \sum_{v\in V_\alpha^1}m(v)w_v(\alpha) - \Big|\sum_{v\in V_\alpha^1}m(v)w_v(\alpha)\big(\frac{\alpha}{|\alpha|}-p\left(f_\alpha(C_v)\right)\big)\Big|\\
	& \geq (1-\varepsilon)\sum_{v\in V_\alpha^1}m(v)w_v(\alpha) \geq  C_\alpha(1-\varepsilon)\sum_{v\in V_\alpha^1}m(v)\\
	& = C_\alpha(1-\varepsilon)m\big(V_\alpha^1\big)
\end{align*}
and
\begin{align*}
	\Big|\sum_{v\in V_\alpha^2}m(v)w_v(\alpha)p\left(f_\alpha(C_v)\right)\Big| \leq C_\alpha\sum_{v\in V_\alpha^2}m(v)=C_\alpha m\big(V_\alpha^2\big).
\end{align*}
The previous two estimates yield
\begin{align*}
	|\Phi(\alpha)| & \geq \Big|\sum_{v\in V_\alpha^1}m(v)w_v(\alpha)p\left(f_\alpha(C_v)\right)\Big|-\Big|\sum_{v\in V_\alpha^2}m(v)w_v(\alpha)p\left(f_\alpha(C_v)\right)\Big|\\
	& \geq C_\alpha\left((1-\varepsilon)m\big(V_\alpha^1\big)-m\big(V_\alpha^2\big)\right)\\
	& \geq C_\alpha\left(m(V)-2\,m\big(V_\alpha^2\big)-\varepsilon\, m(V).\right)
\end{align*}
We shall finally choose
	\[\varepsilon:=\min_{U\in \mathcal{V}_\mathcal{C}}\frac{m(V)- 2\,m(U)}{2\,m(V)}>0.\]
Indeed, \(\varepsilon\) is positive due to the assumption (\ref{eq:TechAssumpDef}). Using \(V_\alpha^2\in \mathcal{V}_\mathcal{C}\) we conclude
	\[|\Phi(\alpha)|\geq \frac{C_\alpha\varepsilon\, m(V)}{2}>0\]
in contradiction to \(\Phi(\alpha)=0\). This completes the proof.
\end{proof}
\subsection{Bounds on planar combinatorial graphs}\label{subsec:EigComb}
We are now in a position to prove the estimate \eqref{eq:intro-estimate-discrete1}:
\begin{theorem}\label{thm:SpecBoundPlanComb}
	Let \(G=(V,E)\) be a finite, simple, connected and planar graph. Let \(m:V\rightarrow(0,\infty)\) be a vertex weight
	\begin{align}		
		2\,(m(u)+m(v))<m(V)\label{eq:TechAssumpPlan}
	\end{align}
	 for any pair of adjacent vertices \(u,v\in V\). Then, for any edge weight \(\mu:E\rightarrow(0,\infty)\), the spectral gap \(\lambda_2(\mathcal L_{m,\mu})\) of the weighted combinatorial Laplacian \(\mathcal{L}_{m,\mu}\) associated with the weights \(m\) and \(\mu\) admits the spectral bound
	 \begin{align}
		\lambda_2(\mathcal{L}_{m,\mu})\leq 8\frac{d_{\max}^\mu}{m(V)}.\label{eq:SpecBoundPlanComb}
	 \end{align}
\end{theorem}
\begin{proof}[Proof of Theorem \ref{thm:SpecBoundPlanComb}]
Since \(G\) is planar and the condition \eqref{eq:TechAssumpPlan} is satisfied, we can apply Lemma \ref{lem:Deformation} to choose a univalent circle packing \(\mathcal{C}=(C_v)_{v\in V}\), so that \(G\) is the intersection graph of \(G\) and
	\[\sum_{v\in V}m(v)p_v=0\]
is satisfied for the centre points \(p_v\) of \(C_v\) for \(v\in V\). For \(v\in V\) let \(r_v\) be the radius of \(C_v\). We define the function \(f:V\rightarrow S^2\) by means of \(f(v)=p_v\). Then, \(f\) is a viable test function in \eqref{eq:CombCourantFischer}. It remains to estimate the quotient \(q(f)/||f||_{\ell_m^2(V;\C^3)}^2.\) First of all, since every centre point \(p_v\) is in \(S^2\), we have
	\[\|f\|_{\ell^2_m(V;\C^3)}^2=\sum_{v\in V}m(v)|p_v|^2=m(V).\]
Moreover, if \(v\) and \(u\) are adjacent, the caps \(C_v\) and \(C_u\) intersect, so \(|p_v-p_u|\leq r_u+r_v\) holds by triangle inequality. Using the Young inequality implies \(|p_v-p_u|^2\leq 2(r_u^2+r_v^2)\) and thus
\begin{align}\label{eq:EstimateWithRadii}
	q(f) =\sum_{e=\{u,v\}\in E}\mu(e)|p_u-p_v|^2 \leq 2\sum_{v\in V}\sum_{e\in E_v}\mu(e) r_v^2\leq 2\Deg{\mu}{\max}\sum_{v\in V}r_v^2.	
\end{align}
To estimate the last term we recall that the surface area of a cap \(C_v\) is equal to \(\pi\,r_v^2\). Furthermore, \(\mathcal{C}\) is univalent, so the sum of the surface areas of the caps in \(\mathcal{C}\) is at most equal to the total surface area of the unit sphere, i.e\ \(\sum_{v\in V}\pi r_v^2\leq 4\pi.\) Plugging this into (\ref{eq:EstimateWithRadii}) yields \(q(f)\leq 8\Deg{\mu}{\max}\) and altogether we obtain
\begin{align*}
	\lambda_2(\mathcal{L}_{m,\mu})\leq \frac{q(f)}{||f||_{\ell_m^2(V;\C^3)}^2}\leq 8\frac{\Deg{\mu}{\max}}{m(V)}.
\end{align*}
This proves the claim.
\end{proof}

\begin{remark}\label{ex:CombExTeckAss}
\begin{remarklist}
\item The estimate \eqref{eq:intro-spielman-teng} by Spielman and Teng in the unweighted case \(\mu\equiv 1\) and \(m\equiv 1\) is included in Theorem \ref{thm:SpecBoundPlanComb}. In that case \eqref{eq:TechAssumpDef} states that \(G\) has at least \(5\) vertices. If \(G\) however has less then \(5\) vertices, \eqref{eq:intro-spielman-teng} is trivially satisfied, since
		\[\lambda_2(\mathcal L)\leq \frac{2|E|}{|V|-1}\leq \frac{d_{\max}|V|}{|V|-1}\leq 8\frac{d_{\max}}{|V|}\]
for \(|V|=2,3,4\). Therefore, Spielman and Teng neither had to impose the condition \eqref{eq:TechAssumpDef} nor had to consider it in their proof in \cite{SpielmanTeng:SpecPartWork07}, whereas it plays an important role in our proof of Lemma \ref{lem:Deformation}.
	\item Note that not every graph satisfies the condition \eqref{eq:TechAssumpPlan}: indeed, consider the edge and vertex weights \(\mu\) and \(m\) associated with the (unweighted) normalized Laplacian on \(G\), i.e.\ \(\mu(e)=1\) and \(m(v)=d_v=|E_v|\) for \(e\in E\) and \(v\in V\). Then \eqref{eq:TechAssumpPlan} means that \(d_v+d_u\leq |E|\) holds for any pair of adjacent vertices \(u,v\in V\) which, for instance, is not satisfied for star graphs with at least two edges.
	\item The following example shows that the condition \eqref{eq:TechAssumpPlan} cannot be omitted: let \(K_4\) be the complete graph with vertex set \(V_{K_4}=\{v_1,v_2,v_3,v_4\}\), constant edge weight \(\mu\equiv 1\) and the vertex weight \(m\)  given by \(m(v_1)=a\) for some constant \(a\geq 1\) and \(m(v_j)=1\) for \(j\neq 1\). The corresponding Laplacian has the eigenvalues
		\(\lambda_1=0,\, \lambda_2=\frac{a+3}{a},\, \lambda_3=4,\, \lambda_4=4\). Note that	\(m(V) = 3+a\) and \(\Deg{\mu}{\max}  =3\), so for an estimate of the form \eqref{eq:SpecBoundPlanComb} to hold, there must exist some constant \(C>0\) so that \(\frac{a+3}{a}\leq C\frac{3}{3+a}\) for all \(a\geq 1\). Obviously such \(C\) does not exist. And, in fact, the condition \eqref{eq:TechAssumpPlan} is not satisfied since \(2(m(v_1)+m(v_2))=2a+2\geq a+3 =m(V)\) holds.
\end{remarklist}
\end{remark}

\subsection{Bounds on planar metric graphs}\label{subsec:EigMetr}
Using the results of the previous section and Corollary we can finally prove our main theorem in the planar setting:
\begin{theorem}\label{thm:maintheorem}
		Let \(\mathcal{G}=(G,\ell)\) be a connected, finite and compact metric graph, whose underlying combinatorial graph \(G=(V,E)\) is planar. Then, the first positive eigenvalue \(\lambda_2(\Delta_\mathcal G)\) of the Kirchhoff Laplacian satisfies the spectral estimate
		\begin{align}\label{eq:SpecEstIntroQG}
			\lambda_2(\Delta_\mathcal G)\leq \frac{C\Deg{\mu}{\max}}{L}.	
		\end{align}		
		Here \(L:=\sum_{e\in E}\ell_e\) is the total length of \(\mathcal{G}\), \(d_{\max}^\mu\) is the maximal weighted degree corresponding to the inverse weight of \(\ell\), that is
		\begin{equation}
			\Deg{\mu}{\max}:=\max_{v\in V}\sum_{e\in E_v}\frac{1}{\ell_e},
		\end{equation}
		and \(C>0\) is a generic constant that does not depend on the metric graph \(\mathcal{G}\). In fact, we may choose \(C=16\pi^2\) for general planar metric graphs and \(C=2\pi^2\) if, additionally, \(\mathcal{G}\) is simple and
	\begin{align}\label{eq:TechAsumpQG}
		\Deg{\ell}{v}+\Deg{\ell}{u}<L
	\end{align}
	holds for every pair of adjacent vertices \(u,v\in V\).
	\end{theorem}
\begin{proof}
	Let us first assume that \(G\) is simple and satisfies the inequality \eqref{eq:TechAsumpQG} for all pairs of adjacent vertices \(u,v\in V\). We consider the combinatorial Laplacian \(\mathcal{L}_{m,\mu}\) associated with the weight functions given by \eqref{eq:InducedWeights}. Note that we have
		\[m(V)=\sum_{v\in V}\Deg{\ell}{v}=2L\]
	 with this choice of weights by the handshaking lemma. We conclude that the inequality \(2(m(v)+m(u))<m(V)\) is equivalent to \(\Deg{\ell}{v}+\Deg{\ell}{u}<L\) for \(u,v\in V\). So, we may apply Theorem \ref{thm:SpecBoundPlanComb} to \(\mathcal{L}_{m,\mu}\). Together with Lemma \ref{eq:CompareSpectralGaps}, we obtain
	 	\[\lambda_2(\Delta_\mathcal G)\leq \frac{\pi^2}{2}\lambda_2(\mathcal{L}_{m,\mu})\leq 4\pi^2 \frac{\Deg{\mu}{\max}}{m(V)}=2\pi^2 \frac{\Deg{\mu}{\max}}{L}.\]
	This proves the statement of Theorem \ref{thm:maintheorem} if \(G\) is simple and \eqref{eq:TechAsumpQG} is satisfied.\\
	Now consider an arbitrary metric graph \(\mathcal{G}\), that is not necessarily simple or does not necessarily satisify \eqref{eq:TechAsumpQG}. Without loss of generality we may assume that \(G\) has at least two edges. Otherwise \(\mathcal{G}\) is either an interval or a loop. Then the spectral gap would be either 
	\begin{align*}
		\begin{array}{ccc}\displaystyle
		\frac{\pi^2}{L^2}=\pi^2\frac{\Deg{\mu}{\max}}{L} & \text{or} &\displaystyle \frac{4\pi^2}{L^2}=2\pi^2\frac{\Deg{\mu}{\max}}{L}
		\end{array}
	\end{align*} and in both cases \eqref{eq:SpecEstIntroQG} holds for \(C=16\pi^2\).\\
	Now, let \(\mathcal{G}'=(G',\ell')\) be the (metric) subdivision graph obtained after dividing each edge of \(\mathcal{G}\) into four edges of equal length (see Figure \ref{fig:subdivision}). We shall write
		\[V_{G'}=V_{\old}\cup V_{\new},\]
	where \(V_{\old}\) is the set of old vertices coinciding with \(V_{G}\) and \(V_{\new}\)  is the set of new vertices that are added on the interior of the edges. Note, that \(G'\) is simple. Moreover, the total length of the graph and the spectral gap of the Kirchhoff Laplacian remain the same after subdividing the graph, since we only add vertices of degree \(2\). For \(v\in V_G\) we have
		\[\Deg{\ell'}{v}=\begin{cases}
			\frac{\Deg{\ell}{v}}{4}, & \text{ if } v\in V_{\old},\\
			\frac{\ell_e}{2}, & \text{ if } v\in V_{\new} \text{ and } v \text{ is on the edge } e.
		\end{cases}\]
Since \(\mathcal{G}\) has at least two edges, this implies \(\Deg{\ell'}{v}+\Deg{\ell'}{u}<L\) for any adjacent vertices \(u,v\in V_{G'}\), thus \(\mathcal{G}\) satisfies the condition \eqref{eq:TechAsumpQG}. We conclude that \(\mathcal{G}'\) satisfies \eqref{eq:SpecEstIntroQG} for \(C=2\pi^2\), i.e.\
	\[\lambda_2(\Delta_{\mathcal G'})\leq 2\pi^2\frac{\Deg{\mu'}{\max}}{L}\]
where \(\mu'\) is the inverse weight of \(\ell'\). Also, note that
	\[\Deg{\mu'}{v}=\begin{cases}
			4 \Deg{\mu}{v}, & \text{ if } v\in V_{\old},\\
			\frac{8}{\ell_e}, & \text{ if } v\in V_{\new} \text{ and } v \text{ is on the edge } e,
		\end{cases}\]
which yields \(\Deg{\mu'}{\max}\leq 8 \Deg{\mu}{\max}\). Using this and the spectral bound for \(\mathcal{G}'\) we conclude
	\[\lambda_2(\Delta_{\mathcal G})=\lambda_2(\Delta_{\mathcal G'})\leq 16\pi^2\frac{\Deg{\mu}{\max}}{L},\]
which completes the proof.
\end{proof}
\begin{figure}
	\tikzstyle{my style}=[circle, draw, fill=black!50,
                        inner sep=0pt, minimum width=4pt]
    \tikzstyle{white style}=[circle, draw, fill=white,
                        inner sep=0pt, minimum width=4pt]
	%
	%
	%
	%
    \begin{minipage}{1\textwidth}
    \centering
    
    \begin{tikzpicture}[thick,scale=0.7]%
    	\draw (2,0) to[out=120,in=-25] (1.2,0.7);
    	\draw (1.2,0.7) node[white style]{} to[out=155,in=0] (0,0.95);
    	\draw (0,0.95) node[white style]{} to[out=180,in=25] (-1.2,0.7);
    	\draw (-1.2,0.7) node[white style]{} to[out=-155,in=60] (-2,0);
    	
		\draw (2,0) to[out=-120,in=25] (1.2,-0.7);
    	\draw (1.2,-0.7) node[white style]{} to[out=-155,in=0] (0,-0.95);
    	\draw (0,-0.95) node[white style]{} to[out=-180,in=-25] (-1.2,-0.7);
    	\draw (-1.2,-0.7) node[white style]{} to[out=155,in=-60] (-2,0);

        \draw (2,0) to[out=60,in=185] (3.3,0.7);
        \draw (3.3,0.7) node[white style]{} to[out=5,in=90] (4.5,0);
        \draw (2,0) to[out=-60,in=-185] (3.3,-0.7);
        \draw (3.3,-0.7) node[white style]{} to[out=-5,in=-90] (4.5,0) node[white style]{};
        
        \draw (-2,0) to[out=120,in=-5] (-3.3,0.7);
        \draw (-3.3,0.7) node[white style]{} to[out=175,in=90] (-4.5,0);
        \draw (-2,0) to[out=-120,in=5] (-3.3,-0.7);
        \draw (-3.3,-0.7) node[white style]{} to[out=-175,in=-90] (-4.5,0) node[white style]{};
        
        \draw (2,0) node[my style]{} -- (1,0);
        \draw (1,0) node[white style]{} -- (0,0);
        \draw (0,0) node[white style]{} -- (-1,0);
        \draw (-1,0) node[white style]{} -- (-2,0) node[my style]{};
          
	\end{tikzpicture}
    \end{minipage}\hfill
    %
    %
    %
    %
    
    \caption{Subdivsion of  a non-simple graph. The old vertices are marked black, whereas the newly added vertices are marked white.}\label{fig:subdivision}
\end{figure}
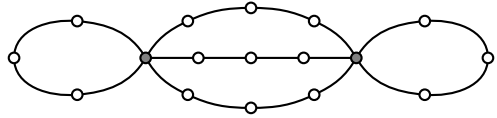
Let us finish this section by considering a number of examples to discuss the spectral estimate \eqref{eq:SpecEstIntroQG} and to compare it with other known results.
\begin{example}[Star graphs]\label{rem:star}
	The spectral gap of the Kirchhoff Laplacian on an equilateral star graph of total length \(L>0\) with \(n\) edges is
		\[\lambda_2(\Delta_{\mathcal{S}_n})=\frac{\pi^2n^2}{4L^2}.\]
	The maximal degree associated to the inverse weight is attained at the center vertex of the star and is given by
		\[\Deg{\mu}{\max}=\frac{n^2}{L},\]
	thus
		\[\lambda_2(\Delta_{\mathcal{S}_n})=\frac{\pi^2\Deg{\mu}{\max}}{4L}.\]
	Therefore, for equilateral star graphs, the growth rate in~\eqref{eq:SpecEstIntroQG} with respect to $\frac{\Deg{\mu}{\max}}{L}$ is correct up to the constant in front of the metric quantities on the right hand side.
\end{example}
\begin{example}[Complete graphs]\label{rem:compl}
	For \(n\in\N\) let \(\mathcal{K}_n\) be the equilateral, complete graph on \(n\) vertices with edge weight \(l\equiv 1\). By Kuratowski's Theorem \cite{Kuratowski:Sur} the underlying combinatorial graph is not planar for \(n\geq 5\). The combinatorial graph has \(\frac{n(n-1)}{2}\) edges and, since the metric graph is equilateral, its total length is equal to the number of edges. Moreover, the degree of every vertex is \(n-1\). The first positive eigenvalue of the associated normalized Laplacian is \(\frac{n}{n-1}\) for \(n\geq 2\), so we may use von Below's formula (\ref{eq:VonBelow}) to derive the spectral gap of the Kirchhoff Laplacian:
		\[\lambda_2(\Delta_{\mathcal{K}_n})=\arccos\left(-\frac{1}{n-1}\right)^2\rightarrow \frac{\pi^2}{4}, \text{ as } n\rightarrow\infty\]
	However,
	\begin{align}\label{eq:EstRohleder}
		\frac{\Deg{\mu}{\max}}{L}=\frac{2}{n}\rightarrow 0, \text{ as } n\rightarrow\infty.
	\end{align}
	Therefore an estimate of the form (\ref{eq:SpecEstIntroQG}) cannot hold for general metric graphs. More precisely, there exists no \(C>0\), so that \eqref{eq:SpecEstIntroQG} holds for any finite, compact and connected metric graph.
\end{example}

\begin{example}[Trees]
	Rohleder \cite{Rohleder:EigEst} has proved that the spectral gap of a compact metric tree \(\mathcal{T}\) -- a compact metric graph without cycles -- admits the estimate
		\[\lambda_2(\Delta_\mathcal{T})\leq \frac{\pi^2}{\diam(\mathcal{T})^2},\]
	where \(\diam(T)\) denotes the metric diameter of the graph, which is the maximal distance between two points in \(\mathcal{T}\). To compare this estimate with our spectral estimate \eqref{eq:SpecEstIntroQG} we restrict ourselves to binary trees of exponential volume growth/decay.\\
	Let \(B_h\) denote the (combinatorial) complete rooted binary tree of height \(h\in\N\), i.e\ \(B_h\) is simple, \(B_h\) has only one vertex \(o\in V_{B_h}\) of degree \(2\), the so called \textit{root} of \(B_h\) and every other vertex \(v\in V_{B_h}\) has either degree \(1\) or degree \(3\) and is connected to \(o\) via a unique path in \(B_h\) whose (combinatorial) length is at most \(h\). The length of this path is called the \textit{generation} \(\gen(v)\) of \(v\). The level of \(v\) is equal to \(h\), if \(v\) has degree \(1\), and it is strictly less than \(h\), if \(v\) has degree \(3\). If an edge \(e\in E_{B_h}\) connects two vertices of generation \(k-1\) and \(k\) we set \(\gen(e)=k\). Note that \(B_h\) has exactly \(2^k\) edges of level \(k\) for \(k=1,\cdots,h\).\\
	We consider two edge weights on \(B_h\). First, let \(\mathcal{B}_h=(B_h,\ell)\) be the associated metric graph of equilateral length \(\ell\equiv 1\). In this case the diameter of \(\mathcal{B}_h\) is \(\diam(\mathcal{B}_h)=2h\), the total length is \(L=2(2^h-2)\) and the maximal degree is \(\Deg{\mu}{\max}=3\), so estimate \eqref{eq:EstRohleder} gives the upper bound \(\lambda_1(\Delta_{\mathcal{B}_h})\leq \frac{\pi^2}{4h^2},\) whereas our estimate \eqref{eq:SpecEstIntroQG}  yields \(\lambda_2(\Delta_{\mathcal{B}_h})\leq \frac{36}{2^h-2}.\)	So, our estimate gives asymptotically a much sharper bound in the equilateral case.\\
	Now, we consider the edge weight \(l\) given by \(l(e)=2^{-\gen(e)}\), i.e.\ \(l\) is of exponential decay with respect to generation of edges. Then we have \(\diam(\mathcal{B}_h)=2(1-2^{-h})\), \(L=h\) and \(\Deg{\mu}{\max}=5\cdot 2^h\). From \eqref{eq:EstRohleder} we obtain \(\lambda_2(\Delta_{\mathcal{B}_h})\leq \frac{\pi^2}{2(1-2^{-h})},\) whereas our bound \eqref{eq:SpecEstIntroQG} leads to \(\lambda_2(\Delta_{\mathcal{B}_h})\leq \frac{120\cdot 2^h}{h}.\) Thus, \eqref{eq:EstRohleder} gives an asymptotically sharper bound for \(l(e)=2^{-\gen(e)}\), which supports the conjecture that the estimate \eqref{eq:SpecEstIntroQG} gives an asymptotically sharper bound, when the edge lengths do not vary too much.
\end{example}
\section{Eigenvalue bounds for graphs of higher genus}\label{sec:boundhigher}
\subsection{The weighted normalized Laplacian}
Let \(G=(V_G,E_G)\) be a simple finite graph and let \(\omega:E\rightarrow(0,\infty)\) be an edge weight. We recall that the normalized Laplacian \(\mathcal L_{\mathrm{norm}}^\omega\) is the operator acting on the space of functions \(f:V_G\rightarrow\mathbb C\) given by
	\[(\mathcal L_{\mathrm{norm}}^\omega f)(v)=\frac{1}{\Deg{\omega}{v}}\sum_{e=\{u,v\}\in E_v}\omega(e)(f(v)-f(u)),\quad v\in V_G,\]
so in the language of Section \ref{subsec:prel-weight-comb-graphs} we have \(\mathcal L_{\mathrm{norm}}^\omega=\mathcal L_{m,\mu}\) with \(\mu(e)=\omega(e)\) for \(e\in E_G\) and \(m(v)=\Deg{\omega}{v}\) for \(v\in V_G\). For our analysis it will be useful to understand the behavior of the spectrum of \(\mathcal L_{\mathrm{norm}}^\omega\) under subdivision of the graph. Let \(G'=(V_{G'},E_{G'})\) be subdivision graph of \(G\) obtained after adding a new vertex on each edge, i.e.\
\begin{align*}
	V_{G'} & =V_G\cup E_G, & E_{G'} & =\left\{\{v,e\}~|~v\in V_G,~e\in E_v\}\right\}
\end{align*}
with the edge weight \(\omega'\) given by
	\[\omega'(\{v,e\})=\omega(e), \quad v\in V_G,~e\in E_v.\]
\begin{lemma}
	Let \(\lambda\in [0,2]\setminus \{1\}\). Then \(\lambda\) is an eigenvalue of \(\mathcal L_{\mathrm{norm}}^{\omega'}(G')\), if and only if \(R(\lambda):=4\lambda-2\lambda^2\) is an eigenvalue of \(\mathcal L_{\mathrm{norm}}^{\omega}(G)\). In that case, \(\lambda\) and \(R(\lambda)\) have the same multiplicity.
\end{lemma}
A proof in the unweighted case (\(\omega\equiv 1\)) is given in \cite{ComellasXieZhang} and it can easily be extended to the weighted case. The lemma immediately yields
	\[\lambda_k(\mathcal L_{\mathrm{norm}}^{\omega}(G))=4\lambda_k(L_{\mathrm{norm}}^{\omega'}(G'))-2\lambda_k(L_{\mathrm{norm}}^{\omega'}(G'))^2,\quad k=1,\ldots, |V_G|,\]
and, in particular,
\begin{equation}\label{eq:ev-estimate-subdivision}
	\lambda_k(\mathcal L_{\mathrm{norm}}^{\omega}(G))\leq 4\lambda_k(L_{\mathrm{norm}}^{\omega'}(G')),\quad k=1,\ldots, |V_G|.
\end{equation}
\subsection{Measured manifolds and the Theorem of Amini and Cohen-Steiner}\label{subsec:ACS-theorem}
Let \(M\) be a closed, oriented, connected and compact smooth surface with a conformal class \(\confclass\) of Riemannian metrics on \(M\). For \(\rmetric\in\confclass\) let \(\mu_\rmetric\) denote the Riemannian measure on \(M\) induced by the metric \(\rmetric\) and let \(\nabla_\rmetric\) denote the gradient defined on the space of smooth functions on \(M\).
\begin{definition}
	Let \(U\subset M\) be an open subset and let \(\mu\) be a finite Radon measure on \(M\) that is absolutely continuous with respect to \(\mu_\rmetric\) for some -- or equivalently all -- \(\rmetric\in\confclass\). For \(k\in \N\) we define the generalized eigenvalues
		\[\lambda_k(U,\mu):=\inf_{\Lambda_k}\sup_{F\in \Lambda_k\setminus\{0\}} \frac{\int_U|\nabla_\rmetric F|^2\diff \mu_\rmetric}{\int_U|F|^2\diff \mu},\]
	where \(\Lambda_k\) varies over the family of \(k\)-dimensional subspaces of \(C^\infty(U)\cap L_\mu^2(U)\).
\end{definition}
We point out that \(\int_U|\nabla_\rmetric F|^2\diff \mu_\rmetric=\int_U|\nabla_\Rmetric F|^2\diff \mu_\Rmetric\) holds for \(\rmetric, \Rmetric\in\confclass\), so the previous definition does not depend on the choice of the Riemannian metric \(\rmetric\in\confclass\). Also, note that
 	\[0=\lambda_1(M,\mu)<\lambda_2(M,\mu) \leq\lambda_3(M,\mu)\leq \ldots \rightarrow \infty\]
are in fact the eigenvalues of the Laplace--Beltrami operator \(\Delta_\rmetric=-\text{div}_\rmetric\circ\nabla_\rmetric\) on \(M\) if we choose \(\mu=\mu_\rmetric\) for \(\rmetric\in\confclass\).
\begin{definition}
	Let \(\mu\) be a finite and absolutely continuous Radon measure on \(M\). A double cover of \(M\) is a finite family \((U_i)_{i\in I}\) of open and connected subsets \(U_i\subset M\), so that for almost every point \(p\in M\) there are two indices \(i_1\neq i_2\) in \(I\) with \(p\in U_{i_1}\cap U_{i_2}\) and \(p\notin U_j\) for \(j\notin \{i_1,i_2\}\). The vicinity graph \(\Gamma\) associated with \((U_i)_{i\in I}\) is the simple graph with vertex set \(I\) where \(i_1\neq i_2\) are adjacent if and only if \(\mu(U_{i_1}\cap U_{i_2})>0\). The measure \(\mu\) induces an edge weight on \(\Gamma\), which by abuse of notation shall be denoted with \(\mu\) as well, given by
		\[\mu(\{i_1,i_2\}):=\mu(U_{i_1}\cap U_{i_2}).\]
	Let \(\mathcal L_{\mathrm{norm}}^{\mu}\) be the associated normalized Laplacian on \(\Gamma\).
\end{definition}
The following theorem by Amini and Cohen-Steiner \cite{AminiCohenSteiner:TransferPrinc} compares the generalized eigenvalues with the eigenvalues of the normalized Laplacian on the vicinity graph.
\begin{theorem}\label{thm:acs-theorem}
	Let \(\mu\) be a finite and absolutely continuous Radon measure on \(M\) and let \((U_i)_{i\in I}\) be a double cover of \(M\). Then we have the eigenvalue estimate
		\[\lambda_k(\mathcal L_{\mathrm{norm}}^{\mu})\leq 2\frac{\lambda_k(M,\mu)}{\eta}, \quad k=1,\ldots, |I|,\]
	where
		\[\eta=\min_{i\in I}\lambda_2(U_i,\mu).\]
\end{theorem}
It was already observed in \cite{AminiCohenSteiner:TransferPrinc} that Theorem \ref{thm:acs-theorem} also holds for double covers of metric graphs and the eigenvalues of the Laplacian on metric graphs. This idea was further developed by Mugnolo and the present author \cite{mugnolo2019lower} and it turns out that, if one chooses suitable double covers of the graph, Theorem \ref{thm:acs-theorem} generally provides sharper lower bounds for the eigenvalues of the Laplacian compared to previously known abstract results.
\subsection{Bounds on embedded combinatorial graphs}\label{subsec:proof-genus}
Our main result of this section will be the following:
\begin{theorem}\label{thm:est-normlapl}
	Let \(G=(V,E)\) be a finite, simple and connected graph of genus \(g\) and let \(\omega:E\rightarrow(0,\infty)\) be an edge weight. Then the corresponding normalized Laplacian satisfies
	\begin{equation}\label{eq:discrete-estimate-genus}
		\lambda_k(\mathcal L_{\mathrm{norm}}^\omega)\leq C\frac{\Deg{\omega}{\max} (g+k)}{\sum_{e\in E} \omega(e)}
	\end{equation}
for all $k=1,\ldots,|V|$, where \(C\) is a generic constant that does not depend on $G$.
\end{theorem}
This theorem improves the main result by Amini and Cohen-Steiner in \cite{AminiCohenSteiner:TransferPrinc} to the weighted case. The main idea of their proof was to construct a double cover on a surface of genus \(g\) that represents the combinatorial structure of \(G\). To adapt their proof we recall this construction.
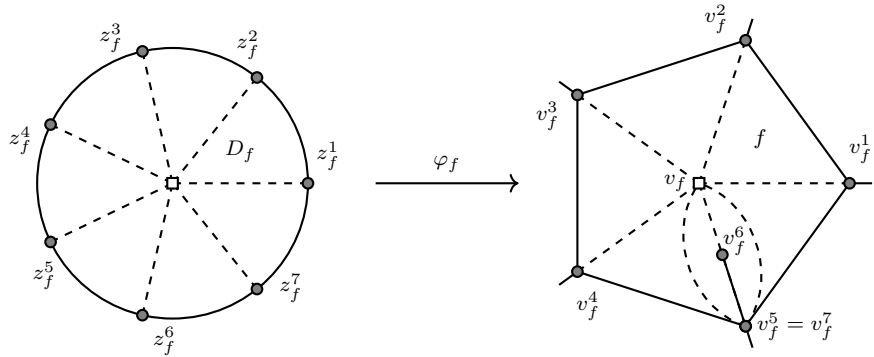
\begin{figure}

	\tikzstyle{mystyle}=[circle, draw, fill=gray,
                        inner sep=0pt, minimum width=4pt]
    \tikzstyle{mystyleR}=[draw, fill=white,
                        inner sep=0pt, minimum width=4pt, minimum height=4pt]
	\begin{tikzpicture}[thick,scale=1]%
	\draw[->] (0:2.7) --node[above]{\tiny\(\varphi_f\)} (0:4.6);
\begin{scope}[xshift=7cm]
	\draw (36:1) node[color=black]{\tiny \(f\)};
  
    	\draw \foreach \x in {0,72,144,...,288} {
        
        (\x:2) edge[color=black] (\x:2.3)
        (\x:2) edge[color=black] (\x+72:2)
        
	};
	
	\draw \foreach \x in {0,72,144,...,288} {(\x:2) edge[dashed,color=black] (0,0)};
	\draw \foreach \x in {0,72,144,...,288} {(\x:2) node[mystyle]{}};
	\draw (288:2) edge[out=54,in=342,dashed,color=black] (0,0);
	\draw (288:2) edge[out=162,in=234,dashed,color=black] (0,0);
	
	\draw (288:1) node[mystyle]{} -- (288:2) node[mystyle]{};
	\draw (0,0) node[mystyleR]{}; \draw (0,0) node[left]{\tiny\(v_f\)};
	\draw \foreach \x in {1,2,3,4} {(72*\x-60:2.2) node{\tiny \(v_f^\x\)}};
	\draw (288:2) node[right]{\tiny \(v_f^5=v_f^7\)};
	\draw (302:0.9) node{\tiny\(v_f^6\)};
	\end{scope}
	
	\draw (25:1) node[color=black]{\tiny \(D_f\)};
%
%
%
	\draw[thick,color=black] (0,0) circle [radius=1.8] ;
	\draw \foreach \x in {0,1,2,3,4,5,6} {(360*\x/7:1.8) edge[dashed,color=black] (0,0)};
	\draw \foreach \x in {1,2,3,4,5,6,7}  {(10+360*\x/7-360/7:2.1) node{\tiny \(z_f^\x\)}};
	\draw \foreach \x in {0,1,2,3,4,5,6}  {(360*\x/7:1.8) node[mystyle]{}};
%
	\draw (0,0) node[mystyleR]{}; 
	\end{tikzpicture}
\caption{Construction of the extended graph \(H\). The white vertex is the added vertex on \(f\) and the dashed edges are the added edges.}\label{fig:const-H}
\end{figure}

Consider an embedding  of \(G\) on a oriented, closed, connected and compact surface \(M\) of genus \(g\). For the sake of notation, we identify the vertices and edges of \(G\) with their drawings on \(M\). Let \(F\) denote the set of (open) faces on \(M\) enclosed by \(G\). Since \(M\) has the same genus as \(G\), every face \(f\in F\) is simply connected (see for instance \cite[Prop. 3.4.1]{MoharThomassen:GraphsOnSurfaces}). Thus, there exists a homeomorphism \(\varphi_f:D_f\rightarrow f\) on the open unit disk that can be extended to surjective and continuous map \(\overline{\varphi_f}:\overline{D_f}\rightarrow \overline{f}\), since the boundary of \(f\) is the finite union of its incident edges in \(E\). Let \(z_f^1,\ldots, z^{N_f}_f\in\partial D_f\) be the preimages of the vertices of \(G\) on the boundary of \(f\) under the surjection \(\overline{\varphi_f}\) appearing in clockwise order and let \(v_f^1,\ldots,v_f^{N_f}\) denote the associated vertices on \(\partial f\). (Note there might be vertices of \(G\) in this list that appear multiple times.) Finally, we can extend the graph \(G\) to obtain a triangulation of \(M\). We add the vertex \(v_f=\varphi_f(0)\) and the edges \(e_f^j=\{v_f,v_f^j\}\) for \(j=1,\ldots,N_f\) that are identified with the images of the lines connecting \(0\) and \(z_f^j\). Let \(H=(V_H,E_H)\) be the graph obtained from the described construction, i.e.\
	\[V_H=V_G\cup \{v_f\}_{f\in F},\quad E_H=E_G\cup\{e_f^j\}_{f\in F}^{j=1,\ldots,N_f}.\]
By construction, the graph \(H\) is embedded in \(M\) and the boundary of each face enclosed by \(H\) consists of three edges of \(H\), where exactly one edge \(e\) belongs to the original graph \(G\). Let \(T_e^i,~i=1,2\) be the two faces that are incident to \(e\); we think of \(T_e^i\) as a triangle in \(M\). However, note that, in general, \(H\) does not define a triangulation of \(M\), as two faces may share multiple edges of \(H\) (see Figure \ref{fig:const-H}), but for our purpose the decomposition of \(M\) given by \(H\) is sufficient to define an appropriate Riemannian metric on \(M\).\\

Following the construction in \cite{Troyanov:Onthemoduli}, we can choose a flat Riemannian metric \(\Rmetric\) on \(M\setminus V_{H}\) with conical singularities in the vertices of \(V_H\) so that every triangle \(T_e^i\) is isometric to an isosceles triangle in the Euclidean plane, where the edges of \(T_e^i\) that do not correspond to \(e\) have length \(1\) and the angle enclosed by \(e\) and the two other edges is
\begin{equation}\label{eq:choice-angle}
	\alpha_e=\frac{1}{2}\arcsin\left(\frac{\pi\omega_e}{2\Deg{\omega}{\max}}\right)
\end{equation}
respectively (see Figure \ref{fig:triangle-h}). For \(v\in V_G\) let \(\theta_v\) be the total angle of the corresponding conical singularity. We have
\begin{equation}\label{eq:angle-sing}
	\theta_v=2\sum_{e\in E_v}\alpha_e=\sum_{e\in E_v} \arcsin\left(\frac{\pi\omega_e}{2\Deg{\omega}{\max}}\right)\leq \sum_{e\in E_v} \frac{\pi\omega_e}{\Deg{\omega}{\max}}=\frac{\pi \Deg{\omega}{v}}{\Deg{\omega}{\max}}\leq \pi,
\end{equation}
for all $v\in V$. Now, one can show that there exists a Riemannian metric \(\rmetric\) on \(M\) that is conformally equivalent to \(\Rmetric\) on \(M\setminus V_H\) (we refer again to \cite{Troyanov:Onthemoduli} for details). Our aim is to apply Theorem \ref{thm:acs-theorem} on \(M\) with the conformal class \(\confclass\) generated by \(\rmetric\) and the Radon measure \(\mu=\mu_\Rmetric\). It remains to define a double cover that represents the combinatorial structure of \(G\).

For every \(e\in E_G\) let \(D_e\) be the diamond
	\[D_e:=T_e^1\cup T_e^2.\]
We decompose \(D_e\) into two triangles by cutting through the diagonal of \(D_e\) perpendicular to \(e\). Every vertex \(v\in V_G\) that is incident to \(e\) is contained in exactly one of the mentioned triangles; let \(S_{e,v}\) be this triangle. Finally, we define the cone
	\[C_v:=\bigcup_{e\in E_v}S_{e,v}.\]
\begin{figure}
\begin{minipage}[t]{0.49\textwidth}
\centering
	\tikzstyle{mystyle}=[circle, draw, fill=gray,
                        inner sep=0pt, minimum width=4pt]
    \tikzstyle{mystyleR}=[draw, fill=white,
                        inner sep=0pt, minimum width=4pt, minimum height=4pt]
	\begin{tikzpicture}[thick,scale=2.5]%
	\coordinate (A) at (1,0);
	\coordinate (B) at (-1,0);
	\coordinate (C) at (0,0.9);
	
	\draw (B) edge[dashed,color=black] node[above]{\small \(1\)} (C);
	\draw (C) node[mystyleR]{} edge[dashed,color=black] node[above]{\small \(1\)} (A);
	
	\draw (B) node[mystyle]{} edge[color=black] node[below]{\small \(e\)} (A);
	\draw (A) node[mystyle]{};
	\draw pic[draw, ->, "\(\alpha_e\)", angle eccentricity=1.7, angle radius = 20] {angle = A--B--C};
	\draw pic[draw, ->, "\(\alpha_e\)", angle eccentricity=1.7, angle radius = 20] {angle = C--A--B};
	\end{tikzpicture}
	\caption{The metric \(\mathfrak h\) on a single triangle \(T_e^i\).}\label{fig:triangle-h}
	\end{minipage}
	\begin{minipage}[t]{0.49\textwidth}
	\centering
	\tikzstyle{mystyle}=[circle, draw, fill=gray,
                        inner sep=0pt, minimum width=4pt]
    \tikzstyle{mystyleR}=[draw, fill=white,
                        inner sep=0pt, minimum width=4pt, minimum height=4pt]
	\begin{tikzpicture}[thick,scale=1.2]%
	\fill [opacity=0.5,blue]
  (30:1) \foreach \x in {60,120,180,240,300}{ -- (30+\x:1) } -- cycle;
  	\fill [opacity=0.5,red]
  (0,0)--(30:1) -- (0:{sqrt(3)}) -- (-30:1) -- cycle;
  \fill [opacity=0.5,orange]
  (0,0)--(30:1) -- (-30:1) -- cycle;
  \fill [opacity=0.5,orange]
  (30:1) -- (0:{sqrt(3)}) -- (-30:1) -- cycle;
  	
    	\draw \foreach \x in {60,120,180,240,300,360} {
        (\x+30:1) edge[dashed,color=black] (\x:{sqrt(3)})
        (\x-30:1) edge[dashed,color=black] (\x:{sqrt(3)})
        (\x+30:1) edge[dashed,color=black] (0,0)
        (\x:{sqrt(3)}) node[mystyle]{} edge[color=black] (0,0)
        
	};
	\draw \foreach \x in {60,120,180,240,300,360} {(\x+30:1) node[mystyleR]{}};
	\draw (0,0) node[mystyle]{} ;
	\draw ({sqrt(3)},0.5) node[color=red!80!black]{\small \(D_e\)};
	\draw (-1.2,0.7) node[color=blue!80!black]{\small \(C_v\)};
	\end{tikzpicture}
	\caption{The covering elements \(D_e\) and \(C_v\).}\label{fig:cover-elements}
	\end{minipage}
\end{figure}

\begin{lemma}
	The family \((D_e)_{e\in E_G}\cup (C_v)_{v\in V_G}\) is a double cover of \(M\). Moreover, two covering elements are adjacent in the corresponding vicinity graph \(\Gamma\), if and only if one element is a diamond \(D_e\) and the other one is a cone \(C_v\), where \(v\) is incident to \(e\) in \(G\). In particular, the vicinity graph \(\Gamma\) is the subdivision graph of \(G\) obtained after adding one vertex on each edge of \(G\). Finally, we have
	\begin{equation}\label{eq:weight-vicinitygraph}
		\mu_{\Rmetric}(D_e\cap C_v)=\mu_\Rmetric(S_{e,v})=\frac{\pi\omega_e}{4\Deg{\omega}{\max}}
	\end{equation}
	if \(e\) is incident to \(v\).
\end{lemma}
\begin{proof}
	The claimed statements follow from the construction of the sets \(D_e, C_v, S_{e,v}\). We point that the volume of \(S_{e,v}\) is given by \eqref{eq:weight-vicinitygraph}, since \(S_{e,v}\) is an isosceles triangle where the two edges of equal length have length \(1\) and the angle enclosed by these two edges is \(2\alpha_e=\arcsin(\frac{\pi\omega_e}{2\Deg{\omega}{\max}})\) by \eqref{eq:choice-angle}.
\end{proof}
We are ready to complete the proof of Theorem~\ref{thm:est-normlapl}. To do so, we apply Theorem \ref{thm:acs-theorem} with the double-cover \((D_e)_{e\in E_G}\cup (C_v)_{v\in V_G}\). Let \(\mathcal L_{\mathrm{norm}}^{\mu_\Rmetric}\) denote the normalized on the vicinity graph \(\Gamma\) with respect to the edge weight induced by \eqref{eq:weight-vicinitygraph}. We obtain
\begin{equation}\label{eq:apply-acs}
	\lambda_k(\mathcal L_{\mathrm{norm}}^{\mu_\Rmetric})\leq 2\frac{\lambda_k(M,\mu_\Rmetric)}{\eta}, \quad k=1,\ldots, |V_G|+|E_G|
\end{equation}
where
		\[\eta=\min\left(\min_{v\in V_G}\lambda_2(C_v,\mu_\Rmetric),\min_{e\in E_G}\lambda_2(D_e,\mu_\Rmetric)\right).\]
As \(\Gamma\) is the subdivision of \(G\) and the edge weight given by \eqref{eq:weight-vicinitygraph} is a scalar multiple of \(\omega\) (the normalized Laplacian is scale-invariant), \eqref{eq:ev-estimate-subdivision} yields
\begin{equation}\label{eq:ev-estimate-subdivision-applied}	
	\lambda_k(\mathcal L_{\mathrm{norm}}^\omega)\leq 4 \lambda_k(\mathcal L_{\mathrm{norm}}^{\mu_\Rmetric}),\quad k=1,\ldots,|V_G|.
\end{equation}
It remains to estimate the right-hand side of \eqref{eq:apply-acs}. We start by cutting \(C_v\) through one the new edges of \(H\), that we previously added to \(G\), we obtain a new open subset \(P_v\) of \(M\). Note that this procedure only decreases \(\lambda_2\). The set \(P_v\) is -- by construction of the Riemannian metric \(\Rmetric\) -- isometric to a polygonal domain in the Euclidean plane. As the angle \(\theta_v\) of the conical singularity \(v\) is less or equal than \(\pi\) (see \eqref{eq:angle-sing}), this polygonal domain is convex (see Figure \ref{fig:unfolding-cone}). Similarly, each \(D_e\) is isometric to an actual diamond in \(\R^2\) which again is convex. Morover, by choice of the metric \(\Rmetric\) all the side lengths of \(P_v\) and \(D_e\) are \(1\), so the diameter of \(P_v\) and \(D_e\) is bounded from above by \(2\). It is known that the spectral gap of the Laplacian (with Neumann boundary conditions) on convex domains of uniformly bounded diameter is uniformly bounded from below (see \cite{PayneWeinberger1960}). More precisely, we obtain
\begin{align}\label{eq:bound-eta}
	\lambda_2(C_v,\mu_\Rmetric) \geq\lambda_2(P_v,\mu_\Rmetric)& \geq \frac{\pi^2}{4}, & \lambda_2(D_e,\mu_\Rmetric) &\geq \frac{\pi^2}{4}
\end{align}
for all \(v\in V_G,~e\in E_G\). Next, we consider \(\lambda_k(M,\mu_\Rmetric)\). A theorem of Hassanezhad \cite[Cor.\ 1.2]{Hassannezhad:Conf} yields
\begin{equation}\label{eq:Hassanezhad-applied}
	\lambda_k(M,\mu_\Rmetric)\leq \tilde C\frac{g+k}{\mu_\Rmetric(M)}.
\end{equation}
for some generic constant \(\tilde C>0\); we point out that the result in \cite{Hassannezhad:Conf} was not stated in the setting of manifolds with a Radon measure, but -- as it was already observed in \cite{AminiCohenSteiner:TransferPrinc} -- the proof given in \cite{Hassannezhad:Conf} can be adapted to this setting. By choice of the metric \(\Rmetric\) we have
\begin{equation}\label{eq:volume-wrt-Rmetric}
	\mu_\Rmetric(M)=\sum_{v\in V_G}\sum_{e\in E_v}\mu_\Rmetric(S_{v,e})=\sum_{v\in V_G}\sum_{e\in E_v}\frac{\pi\omega_e}{4\Deg{\omega}{\max}}=\sum_{e\in E_G}\frac{\pi\omega_e}{2\Deg{\omega}{\max}}
\end{equation}
Plugging \eqref{eq:ev-estimate-subdivision-applied}, \eqref{eq:bound-eta}, \eqref{eq:Hassanezhad-applied} and \eqref{eq:volume-wrt-Rmetric} into \eqref{eq:apply-acs} we obtain
	\[\lambda_k(\mathcal L_{\mathrm{norm}}^\omega)\leq C\frac{\Deg{\omega}{\max}(g+k)}{\sum_{e\in E_G}\omega_e}\]
with \(C=\frac{64\tilde C}{\pi^3}\), which finally proves the claim of Theorem \ref{thm:est-normlapl}.

\begin{remark}\label{rem:on-est-genus}
\begin{remarklist}
	\item If one choose the weights given by \(m(v)=\Deg{\omega}{v}\) and \(\mu(e)=\omega_e\) in Theorem \ref{thm:maintheorem} one obtains the estimate
		\[\lambda_2(\mathcal L_{\mathrm{norm}}^\omega)\leq C\frac{\Deg{\omega}{\max}}{\sum_{e\in E_G}\omega_e}\]
	for planar \(G\). Therefore, Theorem \ref{thm:maintheorem-genus} generalizes Theorem \ref{thm:maintheorem} to higher genus and higher order eigenvalues in the normalized case.
	\item To our knowledge, no explicit upper bounds are known for the constant \(\tilde C\) appearing in Hassanezhad's bound \eqref{eq:Hassanezhad-applied}. This is why we cannot give any explicit upper bounds for the optimal constant \(C>0\) in \eqref{eq:discrete-estimate-genus} or the constants appearing in the following section.
\end{remarklist}
\end{remark}

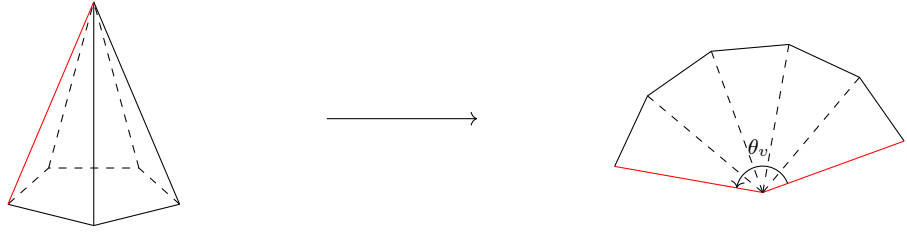
\begin{figure}
\centering
\begin{minipage}{0.33\textwidth}
\tdplotsetmaincoords{70}{0}
\flushright
\begin{tikzpicture}[scale=1.5,tdplot_main_coords]
\def\RI{0.8}
\def\RII{0}

\draw (342:\RI)
  \foreach \x in {342,270,198} { --  (\x:\RI) node at (\x:\RI) (R1-\x) {} };
\draw[dashed] (342:\RI)
  \foreach \x in {60,120,198} { --  (\x:\RI) node at (\x:\RI) (R1-\x) {} };


\foreach \x in {342,270} { \draw (\x:\RI)--(0,5); };
\draw[color=red] (198:\RI)--(0,5);
\foreach \x in {60,120} { \draw[dashed] (\x:\RI)--(0,5); };
\end{tikzpicture}
\end{minipage}\hfill
\begin{minipage}{0.33\textwidth}
	\centering
	\begin{tikzpicture}[scale=1]%
	\draw[->] (-1,0) -- (1,0);
	\end{tikzpicture}
\end{minipage}\hfill
\begin{minipage}{0.33\textwidth}
\flushleft
\begin{tikzpicture}
\draw (20:2)
\foreach \x in {50,80,110,140,170} {--   (\x:2)};

\foreach \x in {20,170} {\draw[color=red] (0,0) -- (\x:2);};
\foreach \x in {50,80,110,140} {\draw[dashed] (0,0) -- (\x:2);};
\coordinate (A) at (20:2);
\coordinate (B) at (0,0);
\coordinate (C) at (170:2);
\draw pic[draw, ->, "\tiny\(\theta_v\)", angle eccentricity=1.7, angle radius = 10] {angle = A--B--C};
\end{tikzpicture}
\end{minipage}
\caption{Cutting and unfolding of the cone \(C_v\) onto the Euclidean plane.}\label{fig:unfolding-cone}
\end{figure}
\subsection{Bounds on embedded metric graphs}
In this final subsection we prove the estimates \eqref{eq:intro-estimate-genus1} and \eqref{eq:intro-estimate-genus2} mentioned in the introduction:
\begin{theorem}\label{thm:maintheorem-genus}
	Let \(\mathcal G=(G,\ell)\) be a connected and compact metric graph of genus \(g\). There exists a generic constant \(C>0\) that does not depend on \(\mathcal G\), so that the ordered eigenvalues \(\lambda_k(\Delta_\mathcal G)\) of the Kirchhoff Laplacian on \(\mathcal G\) satisfy the estimate
	\begin{equation}
		\lambda_k(\Delta_\mathcal G)\leq C\frac{\Deg{\ell}{\max}(g+k)}{\ell_{\min}^2L}, \quad k=1,\ldots, |V|,
	\end{equation}
where \(L:=\sum_{e\in E}\ell_e\) is the total length of \(\mathcal G\), \(\ell_{\min}=\min_{e\in E}\ell_e\) is the minimal length of the edges of \(\mathcal G\), and
	\[d_{\max}^{[\ell]}=\max_{v\in V}d_v^{[\ell]}\]
is the maximal degree of the vertices of \(\mathcal G\).
\end{theorem}
\begin{proof}
	By a subdivision argument similar to the one in the proof of Theorem \ref{thm:maintheorem} we may assume that \(G\) is simple. Let \(\mathcal L_{m,\mu}\) be the combinatorial Laplacian associated with the edge and vertex weights given by \eqref{eq:InducedWeights}. Using the Courant--Fischer Theorem one can show that
		\[\lambda_k(\mathcal L_{m,\mu})\leq \frac{1}{\ell^2_{\mathrm{min}}}\lambda_k(\mathcal L_{\mathrm{norm}}^\ell), \quad k=1,\ldots,|V|.\]
	Thus, Corollary \ref{cor:CompareSpectralGaps} and  Theorem \ref{thm:est-normlapl} yield that
		\[\lambda_k(\Delta_\mathcal G)\leq \frac{\pi^2}{2\ell_{\mathrm{min}}^2}\lambda_k(\mathcal L_{\mathrm{norm}}^\ell)\leq C\frac{d_{\max}^\ell(g+k)}{\ell_{\min}^2L}, \quad k=1,\ldots,|V|.\]
	This proves the claim.
\end{proof}
	\begin{theorem}\label{theo:est-with-dmax}
		Let \(\mathcal G=(G,\ell)\) be a connected and compact metric graph of genus \(g\) and let \(\beta=|E_G|-|V_G|+1\) be its first Betti number. There exists a generic constant \(C>0\), so that
			\begin{equation}\label{eq:est-bettinumber-genus}
				\lambda_k(\Delta_\mathcal G)\leq C\frac{d_{\max}(\beta+k-1)(g+k)}{L^2}
			\end{equation}
		holds for all integers \(k\in\mathbb N\) with \(k\geq \frac{L}{\ell_{\min}}-\beta+1\).
	\end{theorem}
	\begin{proof}
		For given \(k\in\N\) with \(k\geq \frac{L}{\ell_{\min}}-\beta+1\), we set \(n:=k+\beta-1=k+|E_G|-|V_G|\). Then, for \(e\in E\) let \(m_e\) be the unique integer with
			\[m_e-1\leq\frac{n\ell_e}{L}< m_e\]
		and let \(m:=\sum_{e\in E}m_e\). Let \(\tilde{\mathcal G}=(\tilde G,\tilde \ell)\) be the metric subdivision graph obtained after dividing each edge \(e\) into \(m_e\) edges of the same length \(\frac{\ell_e}{m_e}\). Note that, by the condition on \(k\) and the choice of \(n\), we have \(\frac{n\ell_e}{L}\geq 1\). Therefore \(m_e\) has to be at least \(2\). Then, by choice of \(m_e\), we have \(m-|E|\leq n< m\) and
		\begin{equation*}
				\frac{m_e-1}{m_e}\cdot\frac{L}{n}\leq\frac{\ell_e}{m_e}< \frac{L}{n},
		\end{equation*}
		The latter implies
		\begin{equation}\label{eq:estimates-length-subdivision}
			\frac{L}{2n}\leq\frac{\ell_e}{m_e}< \frac{L}{n}
		\end{equation}
		since \(m_e\geq 2\) for \(e\in E\). By construction, the underlying combinatorial graph \(\tilde G\) of \(\tilde{\mathcal G}\) has
			\[|V_{\tilde G}|=|V_G|+m-|E_G|> |V_G|+n-|E_G|=k\]
		vertices. Thus, we obtain
			\[\lambda_k(\Delta_\mathcal G)=\lambda_k(\Delta_{\tilde{\mathcal G}})\leq C\frac{\Deg{\tilde\ell}{\max}(g+k)}{\tilde\ell_{\min}^2 L}\]
		by Theorem \ref{thm:maintheorem-genus}, where \(\tilde\ell_{\min}\) is the minimal length and \(\Deg{\tilde\ell}{\max}\) is the maximal degree of \(\tilde{\mathcal G}\). Using \eqref{eq:estimates-length-subdivision} we obtain \(\ell_{\min}\geq \frac{L}{2n}\) and \(\Deg{\tilde\ell}{\max}\leq d_{\max}\frac{L}{n}\), which again yields
			\[\lambda_k(\Delta_\mathcal G)\leq 4C\frac{d_{\max}n(g+k)}{L^2}=4C\frac{d_{\max}(\beta+k-1)(g+k)}{L^2}.\]
		As \(k\) was arbitrary, this completes the proof.
	\end{proof}

\appendix

\bibliographystyle{alpha}
\bibliography{bib}

\end{document}